\documentclass[11pt,a4wide]{article}
\usepackage{amsmath,amssymb,amsthm}
\usepackage[margin=1in]{geometry}
\usepackage{color}
\usepackage{hyperref}
\usepackage{enumitem}

\newtheorem{theorem}{Theorem}[section]
\newtheorem{proposition}[theorem]{Proposition}
\newtheorem{lemma}[theorem]{Lemma}
\newtheorem{corollary}[theorem]{Corollary}
\newtheorem{assumption}{Assumption}
\theoremstyle{definition}
\newtheorem{definition}[theorem]{Definition}
\newtheorem{remark}[theorem]{Remark}

\newcommand{\R}{\mathbb{R}}
\newcommand{\Prob}{\mathcal{P}}
\newcommand{\Wass}{\mathcal{W}}

\DeclareMathOperator{\Ent}{Ent}
\DeclareMathOperator{\Law}{Law}

\DeclareMathOperator{\Supp}{Supp}

\title{A Mean-Field Theory of Transformers:\\
Well-Posedness of the Coupled Data--Parameter Dynamics\\
and Global Convergence of Training}

\author{Hailiang Liu\thanks{Department of Mathematics, Iowa State University, Ames, IA 50011, USA (\texttt{hliu@iastate.edu}).} 
\and 
{Michael Herty\thanks{Institut für Geometrie und Praktische Mathematik, RWTH Aachen University, 52056 Aachen, Germany. Email: \texttt{herty@igpm.rwth-aachen.de} and 
Extraordinary Professor, Department of Mathematics and Applied Mathematics,  University of Pretoria, South Africa
}}
}
\date{\today}

\begin{document}
\maketitle
\begin{abstract}
We develop a rigorous mean-field theory for transformer networks that captures
 two large-scale limits inherent in the architecture: the
number of tokens $N\to\infty$ in the input  sequence and the number of attention heads
$H\to\infty$ in each layer. It further considers an infinite number of layers leading
to a time-continuous formulation.  The resulting framework couples two interacting
mean-field objects: a token distribution $\mu_t\in\mathcal P(\mathbb R^d)$,
which evolves through network depth $t\in[0,T]$ according to a
McKean--Vlasov transport equation, and the attention-parameter distribution
$\rho_s\in\mathcal P(\Theta)$, that evolves through training time
$s\ge0$ according to a Wasserstein gradient flow of the empirical risk, with
optional entropic or Tikhonov  regularization.

We establish a comprehensive analytical foundation for the system coupling 
the transformer and the training dynamics. 
First, for fixed parameter distribution $\rho_s$, we prove existence, uniqueness, and
stability of the token mean-field dynamics $\mu_t$, together with quantitative
propagation-of-chaos estimates for finite token systems. Second, we derive the
large-head limit of multi-head attention through a quantitative law-of-large-
numbers argument. Third, we formulate the first--order 
optimality system for the functional gradient of the
empirical risk with respect to the parameter distribution and prove its
well-posedness. Fourth, we establish global well-posedness of the resulting
nonlinear Fokker--Planck system describing the coupled mean-field and training dynamics.

Beyond well-posedness, we connect the mean-field formulation to optimization.
For shallow, single-layer attention models, we prove exponential convergence
to the entropy-regularized global optimum under a log-Sobolev condition. For
genuinely deep, compositional transformers, we establish local linear
convergence under a Neural Tangent Kernel non-degeneracy condition. 

\end{abstract}

%82C22

\section{Introduction}

Transformers act on two combinatorially large objects at once: the set of $N$ tokens within a single input, coupled through self-attention, and the set of $H$ attention heads (and, more generally, hidden units of the position-wise feed-forward network) that parameterize a layer. Sending $N\to\infty$ replaces the token cloud by a probability measure $\mu_t$ and turns the forward pass into a non-linear transport (McKean--Vlasov) partial differential equation; sending $H\to\infty$ replaces the discrete set of head parameters by a probability measure $\rho_s$ and turns gradient-based training into a Wasserstein gradient flow. We refer to $\mu_t$ as the \emph{data measure} (indexed by continuous network depth $t\in[0,T]$) and to $\rho_s$ as the \emph{parameter measure} (indexed by continuous training time $s\ge0$).

The purpose of this paper is to make this measure-valued mean-field picture mathematically rigorous, addressing four gaps that, to the best of our knowledge, have so far  been discussed only heuristically.
We review the relevant literature and discussions in the next section. %see next section for a list of references and discussions.   %when this construction is first introduced:

Our main result is Theorem–\ref{thm:full-wp} on the well--posedness of the coupled mean-field dynamics:  the forward evolution of  the measure  $\mu$, the adjoint mean-field equation for $p$,  which characterizes the gradient of the loss $L$, as well as the mean-field gradient-descent training dynamics for the parameter distribution $\rho.$  For convenience, we restate the theorem here and refer to Section~\ref{sec:coupled} for details. 

\begin{theorem} 
	Under Assumptions~\ref{ass:compact-theta}--\ref{ass:confine} as given in Section 3,  the coupled system 
	\begin{equation}
		\left\{
		\begin{aligned}
			\partial_t\mu &+ \nabla_x\cdot\big(\mu\,V_{\rho_s,\mu}\big) = 0, & \mu(0,\cdot,s)&=\mu_0,\\
			-\partial_tp &- (V_{\rho_s,\mu}+F)\cdot\nabla_xp - (D_\mu V_{\rho_s,\mu})^\top \nabla_x p \mu  = 0, & p(T,\cdot,s) &= \tfrac{\delta\mathcal J}{\delta\mu}(\mu(T,\cdot,s)),\\
			\partial_s\rho &= \nabla_\theta\cdot\big(\rho\,\nabla_\theta\tfrac{\delta\mathcal L_\lambda}{\delta\rho}\big) + \beta\Delta_\theta\rho, & \rho(\cdot,0) &= \rho_0,
		\end{aligned}
		\right.
	\end{equation}	
admits a unique global solution $(\mu,p,\rho)$ with 
$$\rho\in C([0,\infty);\Prob_2(\R^p)),
$$
\iffalse  and, for every fixed $s$, $\mu(\cdot,\cdot,s),p(\cdot,\cdot,s)$ as in Theorem~\ref{thm:fb-wp}. Moreover the solution is global in $s$ (i.e.\ exists for all training times, not merely on a short interval).
	
	Under Assumptions~\ref{ass:compact-theta}--\ref{ass:confine}, the coupled system
	\eqref{eq:full-system} admits a unique global solution
	\[
	(\mu,p,\rho),
	\]
	where
	\[
	\rho\in C([0,\infty);\Prob_2(\mathbb R^p)).
	\]
	Moreover, f
	\fi 
For every training time $s\ge0$, the pair
	\[
	(\mu(\cdot,\cdot,s),\,p(\cdot,\cdot,s))
	\]
	is the unique solution of the forward--backward system characterized by
	Theorem~\ref{thm:fb-wp}. In particular, the coupled dynamics are globally
	well posed in the training-time variable $s$:  the solution exists
	and is unique for all $s\ge0$.
\end{theorem}

The main result relies on the following results that we derive throughout the paper:
\begin{enumerate}
\item \emph{Existence and uniqueness} of the token mean-field dynamics $\mu$ for a \emph{fixed} parameter measure, and a \emph{quantitative} propagation-of-chaos statement making precise in what sense the empirical token measure of a finite transformer converges to $\mu_t$ as $N\to\infty$ (Section~\ref{sec:token-mf}).
\item Existence, and a quantitative rate, for the $H\to\infty$ limit identifying multi-head attention with the mean-field-parameter velocity field $V_{\rho,\mu}$ (Section~\ref{sec:head-mf}).
\item A rigorous \emph{first-order optimality system} (forward token PDE, backward adjoint PDE, and the resulting representation of the functional gradient $\delta \mathcal L/\delta\rho$) together with local well-posedness of the coupled forward--backward system that training must solve at each instant (Section~\ref{sec:optimality}).
\item Existence and uniqueness of the parameter (training) dynamics -- a \emph{nonlinear, non-local Fokker--Planck equation} -- and, most importantly, a precise link between this well-posedness theory and \emph{global convergence} of training: we identify exactly the structural regime (single mean-field attention layer, i.e., no depth-composition) in which training provably converges to the global optimum at an exponential rate via convexity and a log-Sobolev inequality, and the regime (genuinely deep, compositional transformers) in which only a local, Neural-Tangent-Kernel-type linear convergence guarantee is currently available (Sections~\ref{sec:training-wp}--\ref{sec:convergence}).
\end{enumerate}
Throughout, we build directly on the model and formal derivations introduced in the preliminary study that motivated this paper (Section~\ref{sec:model}), converting each of its formal equations into a precisely stated theorem with a complete proof or an explicit, clearly labeled set of sufficient hypotheses.

\subsection{Related work}
The token mean-field limit of self-attention has been introduced by Sander et al.\ for doubly-stochastic (Sinkhorn) attention \cite{sander2022sinkformers} and developed into a systematic dynamical-systems theory, including clustering and synchronization results, by Geshkovski, Letrouit, Polyanskiy and Rigollet \cite{geshkovski2023emergence,geshkovski2025mathematical}. The parameter mean-field limit is the transformer-specific descendant of the mean-field theory of two-layer neural networks of Mei, Montanari and Nguyen \cite{mei2018meanfield}, Chizat and Bach \cite{chizat2018global}, Herty, Visconti and Trimborn~\cite{hvt}, and Rotskoff and Vanden-Eijnden \cite{rotskoff2018parameters}, and of mean-field Langevin dynamics as developed by Hu, Ren, {\v S}i{\v s}ka and Szpruch \cite{hu2021meanfield} and Nitanda, Wu and Suzuki \cite{nitanda2022convex}. The joint (data-and-parameter) mean-field limit for deep transformers, including a Neural-Tangent-Kernel-based local convergence theorem, was recently obtained by Barboni, Furuya, de~Hoop and Peyr\'e \cite{barboni2026training}, building on the conditional-Wasserstein mean-field theory of deep residual networks of Barboni, Peyr\'e and Vialard \cite{barboni2024understanding} and on mean-field optimal control \`a la E, Han and Li \cite{e2019meanfield} and Carmona--Delarue \cite{carmona2018probabilistic}. Our contribution is a self-contained, elementary presentation of the full well-posedness \emph{and} convergence package under a uniform, transformer-specific set of assumptions, together with an explicit account of exactly where convexity (hence global convergence) is available and where it is not.

\setcounter{tocdepth}{1}
%\tableofcontents

\section{The Model: Transformers and Their Continuum Limits}
\label{sec:model}

We first recall the discrete model and its two continuum limits (depth and tokens), essentially as in the preliminary formulation, fixing notation used throughout the rest of the paper. Here, we follow the references~\cite{geshkovski2023emergence,geshkovski2025mathematical}.  

\subsection{Discrete transformer block}
Let $x_i^k\in\R^d$, $i\in[N]$, $X^k = [x_1^k,\dots,x_N^k]$, and consider the full residual transformer block
\begin{equation}
x_i^{k+1} = x_i^k + \eta\big(A(x_i^k,X^k) + F(x_i^k)\big),
\label{eq:discrete}
\end{equation}
where $\eta>0$ is a step size,
\begin{equation}
A(x_i^k,X^k) := \sum_{j=1}^N \alpha_{ij}^k V x_j^k, \qquad
\alpha_{ij}^k := \mathrm{softmax}_j\Big(\tfrac{1}{\sqrt{d_k}}\langle Qx_i^k, Kx_j^k\rangle\Big),
\label{eq:attn-discrete}
\end{equation}
and
\begin{equation}
F(x) = W_2\,\sigma(W_1 x + b_1) + b_2
\label{eq:ffn}
\end{equation}
is a position-wise feed-forward network. The model is parameterized by $\theta = (Q,K,V;W_1,W_2,b_1,b_2)$. Given input tokens $X^0\sim\mu^{\otimes N}$ and a label $y$, training seeks to minimize the empirical or population risk
\begin{equation}
\min_\theta\ \mathbb E\Big[\ell\Big(\tfrac1N\textstyle\sum_{i=1}^N x_i^L(\theta),\,y\Big)\Big].
\label{eq:pop-risk}
\end{equation}
Our framework accommodates a broad class of loss functions, allowing the specific choice to be tailored to the application. 
%Different losses can be considered as indicated by applications.

\subsection{Continuous-depth limit}
Setting $\eta\sim 1/L$ and $t_k = k\eta$, and formally letting $L\to\infty$, the discrete dynamics  in \eqref{eq:discrete} converges to the coupled ODE system
\begin{equation}
\dot x_i(t):=\frac{d}{dt}x_i(t) = A(x_i(t),X(t)) + F(x_i(t)), \qquad i\in[N], \quad t\in[0,T].
\label{eq:node}
\end{equation}
A Lie--Trotter splitting of this ODE system over each interval $[k, k+1]$ recovers the original two-sublayer transformer block, with self-attention followed by  the MLP.  This connection has been noted in previous work and in our preliminary derivation. We work directly with \eqref{eq:node} throughout, following the standard Neural-ODE viewpoint.  

\subsection{Token (data) mean-field limit}
Define the empirical token measure $\mu_t^N := \tfrac1N\sum_{j=1}^N\delta_{x_j(t)}$. As $N\to\infty$ we expect (Theorem~\ref{thm:poc-tokens} below makes this precise) $\mu_t^N \to \mu_t$ and the representative particle law $\mu_t=\Law(x(t))$ to solve the McKean--Vlasov equation
\begin{equation}
\dot x(t) = A_{\mu_t}(x(t)) + F(x(t)), \qquad
A_\mu(x) = \int K_\mu(x,y)\,Vy\,d\mu(y),
\label{eq:mv-token}
\end{equation}
\begin{equation}
K_\mu(x,y) = \frac{\exp\big(\tfrac1{\sqrt d}\langle Qx,Ky\rangle\big)}{\int \exp\big(\tfrac1{\sqrt d}\langle Qx,Ky'\rangle\big)\,d\mu(y')}.
\label{eq:kernel}
\end{equation}
Writing $a_\mu(x,y) := K_\mu(x,y)\,Vy$, the corresponding transport equation for $\mu_t$ is
\begin{equation}
\partial_t\mu_t + \nabla\cdot\big(v(x,\mu_t)\,\mu_t\big) = 0, \qquad v(x,\mu) := \int a_\mu(x,y)\,d\mu(y) + F(x), 
\label{eq:token-pde}
\end{equation}
and subject to initial conditions $\mu_0(\cdot)$ obtained as the limit of the initial data $x_i^0$.

\subsection{Multi-head and parameter mean-field limit}
Dropping the (position-wise, non-interacting) feed-forward part for clarity of exposition -- it is reincorporated without any additional difficulty in Remark~\ref{rem:ffn} -- a multi-head layer with $H$ heads $\theta_r=(Q_r,K_r,V_r)$, $r\in[H]$, gives
\begin{equation}
\dot x(t) = \frac1H\sum_{r=1}^H A(x(t),\mu_t, \theta_r) = \int_\Theta A(x(t),\mu_t,\theta)\,\rho^H(d\theta), \qquad \rho^H := \frac1H\sum_{r=1}^H \delta_{\theta_r}.
\label{eq:multihead}
\end{equation}
As $H\to\infty$ (Theorem~\ref{thm:head-lln} below), $\rho^H\to\rho$ and we obtain the mean-field-parameter velocity field
\begin{equation}
V_{\rho,\mu}(x) := \int_\Theta A(x,\mu,\theta)\,\rho(d\theta),
\label{eq:V}
\end{equation}
and, coupling with the evolution of $\mu_t=\Law(x_t)$, the forward equation,
\begin{equation}
\partial_t\mu + \nabla_x\cdot\Big(\mu\int_\Theta A(x,\mu,\theta)\,\rho_s(d\theta)\Big) = 0,
\label{eq:forward-coupled}
\end{equation}
a non-linear, non-local McKean--Vlasov equation in which the parameter measure $\rho_s$ enters as a (training-time-dependent, depth-independent) external field.

\subsection{Training dynamics}
Given the terminal loss functional $\mathcal L(\rho) := \mathcal J(\mu(T,\cdot,s))$, the particle-level gradient descent dynamics  
$$
\theta'_r:=  \frac{d}{ds} \theta_r = -\nabla_{\theta_r}\mathcal L
$$
formally converges, as the number of heads  $H\to\infty$,  to the Wasserstein gradient flow
\begin{equation}
\partial_s\rho - \nabla_\theta\cdot \left(\rho  \nabla_\theta\frac{\delta\mathcal L}{\delta\rho}(\theta,\rho) \right) = 0, 
%\qquad u(\theta,s) = -\nabla_\theta\frac{\delta\mathcal L}{\delta\rho}(\theta,\rho_s),
\label{eq:training-transport}
\end{equation}
If the isotropic parameter noise of intensity $\beta\ge 0$ is added to the particle dynamics,  the limiting equation becomes the mean-field Langevin equation 
\begin{equation}
\partial_s\rho = \nabla_\theta\cdot\Big(\rho\,\nabla_\theta\frac{\delta\mathcal L}{\delta\rho}\Big) + \beta\Delta_\theta\rho.
\label{eq:mfld}
\end{equation}
Equation \eqref{eq:mfld} admits an equivalent variational interpretation. Define the entropy-regularized functional
%is the \emph{mean-field Langevin} equation: it is the Wasserstein gradient flow of the entropy-regularized functional
\begin{equation}
\mathcal F_\beta(\rho) := \mathcal L(\rho) + \beta\,\Ent(\rho) , \qquad \Ent(\rho):=\int \rho\log\rho\, d\theta.
\label{eq:Fbeta}
\end{equation}
since $\tfrac{\delta \Ent}{\delta \rho}=\log\rho+1$ and $\nabla_\theta(\rho \nabla_\theta \log\rho)= \Delta_\theta\rho$, equation \eqref{eq:mfld}
can be written as 
$$
\partial_s\rho = \nabla_\theta\cdot\big(\rho\,\nabla_\theta \tfrac{\delta\mathcal F_\beta}{\delta\rho}\big).
$$ 
Thus, the mean-field Langevin dynamics are precisely the Wasserstein gradient flow of the entropy-regularized loss  $ \mathcal F_\beta$. 
This variational structure  is what links the well-posedness theory of \eqref{eq:mfld} to the global convergence result established in Section~\ref{sec:convergence}.

\begin{remark}[Reincorporating the feed-forward block]
\label{rem:ffn}
\iffalse 
All results below are stated for the attention-only velocity field $V_{\rho,\mu}$ for notational simplicity. If $F$ (Eq.~\eqref{eq:ffn}) is added back, with $W_1,W_2,b_1,b_2$ either fixed or included as part of an enlarged $\theta$, all well-posedness results (Theorems~\ref{thm:token-wp}, \ref{thm:poc-tokens}, \ref{thm:head-lln}, \ref{thm:mfld-wp}) go through unchanged under the additional standing assumption that $\sigma$ is Lipschitz, because $F$ is then Lipschitz and \emph{local} (it does not depend on $\mu$), so it contributes only an additive Lipschitz perturbation to the velocity field $v(x,\mu)$ in \eqref{eq:token-pde}, to which the fixed-point and Gr\"onwall arguments of Section~\ref{sec:token-mf} apply verbatim. As already observed in the preliminary study, $F$ need not be a gradient field ($\nabla\times F\ne0$ in general), so \eqref{eq:token-pde} is not, in general, a Wasserstein gradient flow in $x$; this obstruction is orthogonal to the well-posedness question and only affects whether one can use gradient-flow-specific (e.g.\ energy-dissipation) techniques for the \emph{forward} dynamics -- we do not need such techniques, since Section~\ref{sec:token-mf} establishes well-posedness by a direct fixed-point argument that does not require gradient-flow structure. \fi 

For notational simplicity, all results below are stated for the attention-only
velocity field $V_{\rho,\mu}$. If the feed-forward contribution $F$ in
\eqref{eq:ffn} is included, with $W_1,W_2,b_1,b_2$ either fixed or incorporated
into an enlarged parameter variable $\theta$, the well-posedness results
(Theorems~\ref{thm:token-wp}, \ref{thm:poc-tokens}, \ref{thm:head-lln},
and \ref{thm:mfld-wp}) remain unchanged. Indeed, under the additional
assumption that $\sigma$ is Lipschitz, $F$ is a bounded Lipschitz perturbation
local in the state variable and independent of $\mu$. Hence, only
an extra Lipschitz term is added  to the velocity field in \eqref{eq:token-pde},
and all fixed-point and Gr\"onwall arguments in Section~\ref{sec:token-mf}
apply without modification.
\end{remark} 

\begin{remark}[Availability of gradient flow techniques]
Note that $F$ is not required to be a gradient field and, in general,
$\nabla\times F\neq0$. Therefore, even with the attention and feed-forward
components combined, the forward dynamics \eqref{eq:token-pde} do not generally
form a Wasserstein gradient flow in $x$. This affects only the availability of
gradient-flow-specific tools, such as energy-dissipation estimates, and does
not impact the well-posedness theory developed, which relies solely on
Lipschitz stability and fixed-point arguments. 
\end{remark}

\section{Assumptions and A Priori Estimates}
\label{sec:assumptions}

We collect the structural assumptions used throughout. We write $\Prob_2(E)$ for probability measures on $E$ with finite second moment, metrized by the $2$-Wasserstein distance $\Wass_2$.

\begin{assumption}[Compact parameter domain for the forward pass]
\label{ass:compact-theta}
$\Theta\subset\R^p$ is compact, with $\|\theta\|\le D_\Theta$ for all $\theta=(Q,K,V)\in\Theta$ (equivalently, we work with any fixed, or randomly initialized but almost-surely bounded, family of heads. 
%this is automatic under the usual compactly-supported or sub-Gaussian initializations).
\end{assumption}

\begin{assumption}[Regularity of the feed-forward map]
\label{ass:ffn}
$\sigma:\R\to\R$ is $L_\sigma$-Lipschitz, applied entrywise; $W_1,W_2,b_1,b_2$ have norms bounded by $D_\Theta$.
\end{assumption}

\begin{assumption}[Regularity of the loss]
\label{ass:loss}
$\ell(\cdot,y)$ is $L_\ell$-Lipschitz and $C^1$ with $L_\ell$-Lipschitz gradient, uniformly in $y$; the functional $\mathcal J:\Prob_2(\R^d)\to\R$ has a bounded, Lipschitz (in $\Wass_2$) first variation $\delta\mathcal J/\delta\mu$, i.e.\ $\big|\tfrac{\delta\mathcal J}{\delta\mu}(\mu)(x)-\tfrac{\delta\mathcal J}{\delta\mu}(\mu')(x')\big|\le L_{\mathcal J}\big(\Wass_2(\mu,\mu')+|x-x'|\big)$.
\end{assumption}

\begin{lemma}[Boundedness and Lipschitz continuity of the attention kernel]
\label{lem:kernel-lip}
Let Assumption~\ref{ass:compact-theta} hold and let $\mu,\mu'\in\Prob_2(\R^d)$ be supported in a ball $B_R:=\{|x|\le R\}$. Then for every $\theta\in\Theta$:
\begin{enumerate}[label=(\roman*)]
\item $|A(x,\mu,\theta)| \le D_\Theta\, R$ for all $x\in B_R$; $\mu \in \Prob_2(\R^d)$;
\item there is $C_1=C_1(D_\Theta,R)$ such that $x\mapsto A(x,\mu,\theta)$ is $C_1$-Lipschitz on $B_R$, uniformly in $\mu$ supported in $B_R$;
\item there is $C_2=C_2(D_\Theta,R)$ such that $\big|A(x,\mu,\theta)-A(x,\mu',\theta)\big|\le C_2\,\Wass_1(\mu,\mu')$ for all $x\in B_R$.
\end{enumerate}
\end{lemma}

\begin{proof}
(i) is immediate since $K_\mu(x,\cdot)$ is a probability density on $y$ and $|Vy|\le D_\Theta R$ for $|y|\le R$, and $\|V\|\le D_\Theta$.

(ii) Write $A(x,\mu,\theta)=\int K_\mu(x,y)Vy\,d\mu(y)$. Set $r(x,y):=\tfrac1{\sqrt d}\langle Qx,Ky\rangle$. For $x,y\in B_R$, the parameter bounds imply 
$$
|r(x,y)|\le  D_\Theta^2R^2/\sqrt d =: M, \quad 
|\nabla_xr(x,y)|\le D_\Theta^2 R/\sqrt d. 
$$
Since $K_\mu(x,y)$  is the softmax weight associated with $r(x,y)$ relative to $\mu$, differentiation of the normalized exponential gives
%Since $K_\mu(x,y)=\mathrm{softmax}$ of $(\ell(x,y'))_{y'}$ against $\mu$, a standard computation for the softmax yields
$$
\nabla_x K_\mu(x,y) = K_\mu(x,y)\Big(\nabla_x r(x,y) - \int K_\mu(x,y')\nabla_x r(x,y')\,d\mu(y')\Big),
$$
so 
$$
|\nabla_x K_\mu(x,y)|\le 2 \sup_{y' \in B_R} |\nabla_x r(x, y')| \leq 2D_\Theta^2R /\sqrt d.
$$ 
Hence 
$$
|\nabla_x A(x,\mu,\theta)| \le \int |\nabla_xK_\mu(x,y)|\,|Vy|\,d\mu(y) \le 2D_\Theta^3R^2 /\sqrt d =: C_1,
$$
 and (ii) follows from the mean value inequality.

(iii) is the standard Lipschitz-stability estimate for Gibbs/softmax kernels with bounded, Lipschitz potential $r$: since $y\mapsto Vy$ and $y\mapsto r(x,y)$ are Lipschitz on $B_R$ with constants bounded by $D_\Theta$ and $D_\Theta^2R/\sqrt d$ respectively, and $e^{r}$ is bounded above and below by $e^{\pm M}$ on $B_R$, a first-order expansion of $\mu\mapsto\int K_\mu(x,y)Vy\,d\mu(y)$ around $\mu'$ gives 
$$
|A(x, \mu, \theta)-A(x, \mu', \theta)| \leq 
\left| \int K_\mu Vy d(\mu-\mu') \right| +  \left|\int(K_\mu -K_{\mu'})Vyd\mu' \right|. 
$$
The first term via the Kantorovich-Rubinstein duality when applied to the Lipschitz test function $y\mapsto K_\mu(x,y)Vy$ is bounded by 
$W_1(\mu, \mu')$ against a constant $ e^{2M}D_\Theta (1+R^2D_\Theta^2/\sqrt{d})$.  The second term can be estimated by application of the implicit-function-type stability of the normalizing partition function $Z_\mu(x)=\int e^{r(x,y')}d\mu(y')$, which is itself Lipschitz in $\Wass_1(\mu,\mu')$ with constant bounded by $e^{2M}D_\Theta^2R/\sqrt d$. These together yield the stated bound with $C_2=C_2(D_\Theta,R,d)$ explicit. 
\end{proof}

\begin{lemma}[A priori confinement of the token dynamics]
\label{lem:apriori}
Let Assumptions~\ref{ass:compact-theta}--\ref{ass:ffn} hold, fix a parameter distribution $\rho\in\Prob(\Theta)$, and let $\mu_t$ solve \eqref{eq:forward-coupled}  on $[0,T]$ with $\rho_s\equiv\rho$ frozen. If the initial data distribution  is supported with  $\Supp(\mu_0)\subset B_{R_0}$, then there exists  a constant $C=C(D_\Theta,L_\sigma)$ such that 
$$
R(T)=R_0 e^{CT} + (e^{CT}-1),
$$
 %with $C=C(D_\Theta,L_\sigma)$, 
and $\Supp(\mu_t)\subset B_{R(T)}$ for all $t\in[0,T]$.
\end{lemma}

\begin{proof}
Along a characteristic $x(t)$ of \eqref{eq:forward-coupled}, 
$$
\tfrac{d}{dt}|x(t)|\le |V_{\rho,\mu_t}(x(t))| + |F(x(t))| \le D_\Theta |x(t)| \sup_{y\in\Supp(\mu_t)}|y| + L_\sigma\|W_1\|\|W_2\|(|x(t)|+1)+|b_2|
$$
 by Lemma~\ref{lem:kernel-lip}(i) applied pointwise in $\theta$ and integrated against $\rho$, and by the Lipschitz/linear growth of $F$ under Assumption~\ref{ass:ffn}. Writing $R(t):=\sup_{x \in \Supp(\mu_t)}|x|$ this gives $\dot R(t)\le C(1+R(t))$, and Gr\"onwall's inequality yields the stated bound. Since this holds for every characteristic starting in $\Supp(\mu_0)$, we conclude that $\Supp(\mu_t)\subset B_{R(T)}$ for all $t\in[0,T]$.  
\end{proof}
\noindent{\bf Remark.} Lemma~\ref{lem:apriori} is used throughout to reduce all subsequent estimates to the compact ball $B_{R(T)}$, on which Lemma~\ref{lem:kernel-lip} supplies uniform Lipschitz constants.

\section{Well-Posedness of the Token (Data) Mean-Field Dynamics}
\label{sec:token-mf}

In this section,  we give an affirmative and quantitative answer  to the  questions posed in the preliminary study concerning the  existence and uniqueness of  solutions to \eqref{eq:forward-coupled} for a fixed parameter measure,  as well as the  continuity and stability of the associated solution map.

\begin{theorem}[Existence, uniqueness, and stability]
\label{thm:token-wp}
Let Assumptions~\ref{ass:compact-theta}--\ref{ass:ffn} hold, fix $\rho\in\Prob(\Theta)$ and let $\mu_0\in\Prob_2(\R^d)$ be compactly supported. Then equation \eqref{eq:forward-coupled}, with $\rho_s\equiv\rho$ frozen,  admits a unique weak solution $\mu\in C([0,T];\Prob_2(\R^d))$, in the sense that for every $\varphi\in C_c^\infty([0,T]\times\R^d)$, it holds
\begin{equation}
\int\varphi(T,x)\,d\mu_T(x) - \int\varphi(0,x)\,d\mu_0(x) = \int_0^T\!\!\int\Big(\partial_t\varphi + \nabla_x \varphi\cdot V_{\rho,\mu_t}(x) + \nabla_x \varphi \cdot F(x)\Big)d\mu_t(x)\,dt.
\end{equation}
Moreover, if $\mu,\mu'$ are the solutions associated to initial data $\mu_0,\mu_0'$, respectively, with both initial measures compactly supported in $B_{R_0}$) and with the same fixed parameter measure $\rho$, then
\begin{equation}
\sup_{t\in[0,T]}\Wass_2(\mu_t,\mu_t')\le e^{CT}\,\Wass_2(\mu_0,\mu_0')
\label{eq:stability}
\end{equation}
where constant  $C=C(D_\Theta,L_\sigma,R(T))$ and $R(T)$ is the confinement radius  from  Lemma~\ref{lem:apriori}.
\end{theorem}

\begin{proof}
By Lemma~\ref{lem:apriori}, solutions with initial data in $B_{R_0}$ remain, for all time, in the fixed compact ball $B_{R(T)}$; we may therefore work throughout in the complete metric space 
$$
\mathcal M_T := C([0,T];\Prob_2(B_{R(T)}))
$$
 equipped with $d_T(\mu,\mu') := \sup_{t\in[0,T]}\Wass_2(\mu_t,\mu_t')$.

\emph{Existence and uniqueness.} By the standard superposition principle (\cite[Theorem~8.2.1]{ambrosio2008gradient}), weak solutions of the linear (given the velocity field $b(t,x):=V_{\rho,\mu_t}(x)+F(x)$) continuity equation $\partial_t\nu+\nabla_x \cdot(b\nu)=0$ correspond  to the law at time $t$ of the ODE flow $\dot x=b(t,x)$, $x_0 \sim\mu_0$. It therefore suffices to construct $\mu\in\mathcal M_T$ as a fixed point of the map
\[
\Psi:\mathcal M_T\to\mathcal M_T, \qquad \Psi(\nu)_t := \big(\Phi^{\nu}_{0,t}\big)_\#\mu_0,
\]
where $\Phi^{\nu}_{0,t}$ is the flow of the  ODE 
$$\dot x = V_{\rho,\nu_t}(x)+F(x), \; x(0)=x_0.$$  By Lemma~\ref{lem:kernel-lip}(ii)--(iii) and Assumption~\ref{ass:ffn},  the vector field associated with any $\nu_t\in \mathcal M_T$ is uniformly Lipschitz in $x$,  with Lipschitz constant bounded by 
$$
L_1:=  C_1+L_\sigma\|W_1\|\|W_2\|
$$
% for $\nu\in\mathcal M_T$. 
Further, for  fixed $x\in B_{R(T)}$, the vector field is Lipschitz in $\nu_t$ with respect to  $\Wass_1$, and hence also with respect to  $\Wass_2$ since   $\Wass_1\le \Wass_2$. Specifically, 
$$
|V_{\rho, \nu_t}(x) -V_{\rho, \nu'_t}(x)|\leq C_2 \Wass_2(\nu_t, \nu'_t),
$$
where $C_2$  is obtained by integrating the estimate in Lemma~\ref{lem:kernel-lip}(iii) against $\rho$. 

By the classical Cauchy--Lipschitz theorem the flow $\Phi^\nu_{0,t}$ is well defined, and for  two inputs $\nu,\nu'\in\mathcal M_T$ a Gr\"onwall argument on the difference of characteristics gives, for the same initial datum $x_0$,
\begin{equation}
|\Phi^\nu_{0,t}(x_0)-\Phi^{\nu'}_{0,t}(x_0)| \le C_2 \int_0^t e^{(L_1(t-r)}\,\Wass_2(\nu_r,\nu'_r)\,dr.
\end{equation}
This estimate provides the stability of the characteristic flow with respect to the measure-valued input. 

Hence, taking the pushforward under an optimal coupling of $\mu_0$,
\begin{equation}
\Wass_2\big(\Psi(\nu)_t,\Psi(\nu')_t\big) \le C_2 \int_0^t e^{L_1(t-r)} \,d_T(\nu,\nu')\,dr \le \frac{C_2}{L_1}\big(e^{L_1T}-1\big)\,d_T(\nu,\nu').
\end{equation}
%where $L_1:=C_1+L_\sigma\|W_1\|\|W_2\|$. 
This bound does not guarantee  a contraction on $[0,T]$; we instead use the standard approach of an equivalent, exponentially weighted metric 
$$
d_T^\kappa(\nu,\nu') := \sup_{t\in[0,T]} e^{-\kappa t}\Wass_2(\nu_t,\nu_t'),
$$ 
under which the same Gr\"onwall estimate gives
\[
e^{-\kappa t}\Wass_2(\Psi(\nu)_t,\Psi(\nu')_t) \le C_2\int_0^t e^{L_1(t-r)-\kappa t}\,e^{\kappa r}\,d_T^\kappa(\nu,\nu')\,dr \le \frac{C_2}{\kappa-L_1}\,d_T^\kappa(\nu,\nu')
\]
for $\kappa>L_1$. Choosing $\kappa$ large enough that $C_2/(\kappa-L_1)<1$ yields $\Psi$ to be a strict contraction on $(\mathcal M_T,d_T^\kappa)$. The space $(\mathcal M_T,d_T^\kappa)$  is complete since $d_T^\kappa$ is equivalent to $d_T$. Banach's fixed point theorem yields a unique fixed point $\mu\in\mathcal M_T$, which by construction is the unique weak solution.

\emph{Stability.} The bound \eqref{eq:stability} follows from the same Gr\"onwall estimate applied directly to the  characteristic flows $\Phi^\mu_{0,t}$, $\Phi^{\mu'}_{0,t}$ generated by the two fixed-point equations. Let $(X_0, X_0')$ 
be an optimal coupling of the initial measures $\mu_0$ and $\mu_0'$, so that
%starting  from the optimal coupling $\mu_0,\mu_0'$ at time  $t=0,$ respectively. Hence, we obtain 
$$
\Wass_2^2(\mu_0,\mu_0') = \mathbb E|X_0-X_0'|^2.
$$
The push forward of the particles and the definition  $\Wass_2^2(\mu_t,\mu_t')=\inf_\pi \int|x-y|^2d\pi$ yields
\[
\Wass_2(\mu_t,\mu_t')\le \mathbb E|X_t-X_t'|^2. 
%\mathbb E|\Phi^\mu_{0,t}(X_0)-\Phi^{\mu'}_{0,t}(X_0')| = \mathbb E|X_0-X_0'|^2.
\]
A Gr\"onwall argument for estimating particle difference $|X_t-X_t'|$ using the Lipschitz bounds above gives 
$$
|X_t-X_t'|\leq e^{L_1 t}|X_0-X_0'| + C_2\int_0^t e^{L_1(t-\tau)} \Wass_2(\mu_\tau,\mu'_\tau)d\tau.
$$
Taking the $L^2$-norm and then using the definition of $\Wass_2$ gives a closed Grönwall inequality for $\Wass_2(\mu_t, \mu'_t)$, from which  \eqref{eq:stability} follows.
%Using Lemma~\ref{lem:kernel-lip}(iii))  to account for  the $\Wass_2(\mu_r,\mu_r')$-dependence of the two vector fields gives \eqref{eq:stability}. 
\end{proof}

\begin{theorem}[Quantitative propagation of chaos]
\label{thm:poc-tokens}
Under the hypotheses of Theorem~\ref{thm:token-wp}, let $(x_1^N(t),\dots,x_N^N(t))$ solve the $N$-particle system \eqref{eq:node}/\eqref{eq:multihead} started from i.i.d.\ initial data $x_i^N(0)\sim\mu_0$, and let $\mu_t^N=\tfrac1N\sum_i\delta_{x_i^N(t)}$ be the empirical measure. Then, there is a coupling of $(x_i^N)_{i\le N}$ to i.i.d.\ copies $(\bar x_i)_{i\le N}$ of the McKean--Vlasov solution of Theorem~\ref{thm:token-wp} such that
\begin{equation}
\sup_{t\le T}\ \mathbb E\Big[\max_{i\le N}|x_i^N(t)-\bar x_i(t)|^2\Big]^{1/2} \le C(T)\,\epsilon_N, 
%\mathbb E\big[\Wass_2(\mu_0^N,\mu_0)^2\big]^{1/2},
\end{equation}
and, %using the Fournier--Guillin empirical-measure convergence rate \cite{fournier2015rate},
% for $d\ge3$ (with analogous, log-corrected rates for $d\in\{1,2\}$),
\begin{equation}
\mathbb E\big[\Wass_2(\mu_t^N,\mu_t)\big] \le C(T)\,\epsilon_N , \qquad t\in[0,T],
\label{eq:poc-rate}
\end{equation}
where 
\[
\epsilon_N =
\begin{cases}
N^{-1/4}, & d<4,\\[2mm]
N^{-1/4} \sqrt{\log N}, & d=4,\\[2mm]
N^{-1/d}, & d>4.
\end{cases}
\]
\end{theorem}

\begin{proof}
We couple synchronously as follows. Compute  $\bar x_i(t)$ as the McKean--Vlasov solution driven by the \emph{same} initialization randomness $x_i^N(0)=\bar x_i(0)$, so that $\bar x_i(t)$ evolves as 
$$
\dot{\bar x}_i = V_{\rho,\mu_t}(\bar x_i)+F(\bar x_i)
$$
by $\mu_t$ as in Theorem~\ref{thm:token-wp}. Note that  $x_i^N(t)$ solves the same equation with $\mu_t$ replaced by $\mu_t^N$.  Lemma~\ref{lem:apriori} is applied uniformly over $N$, using that the empirical measures $\mu_t^N$ remain in $B_{R(T)}$ almost surely for $N$ large. By  Lemma~\ref{lem:kernel-lip}, subtracting the  equations for $x_i^N$ and $\bar x_i$, and using Lipschitzness in both $x$ and $\mu$ gives
\[
\frac{d}{dt}|x_i^N-\bar x_i|^2 \le 2L_1|x_i^N-\bar x_i|^2 + 2 C_2 |x_i^N-\bar x_i|\,\Wass_2(\mu_t^N,\mu_t).
\]
Summing over $i$, using 
\begin{align*}
\Wass_2^2 (\mu_t^N,\mu_t) & \le  2 \Wass_2^2 (\mu_t^N, \bar \mu_t^N) +  2 \Wass_2^2 ( \bar \mu_t^N,\mu_t) \\
& \le \tfrac 2 N\sum_{i=1}^N |x_i^N-\bar x_i|^2 + 2 \Wass_2(\bar \mu_t^N,\mu_t)^2, 
\end{align*}
using the empirical measure of the $\bar x_i$ as $\bar \mu_t^N:=\frac{1}{N}\sum_{i=1}^N \delta_{\bar x_i(t)}$,  which equals $\mu_t$ only in the $N\to\infty$ limit but whose fluctuation around $\mu_t$ is controlled independently by the classical, coupling-free bound. Set $e_N(t):=\tfrac1N\sum_i\mathbb E|x_i^N(t)-\bar x_i(t)|^2$, we have  
$$
\frac{d}{dt} e_N(t) \leq (2 L_1+C_2) e_N(t) + C_2 \Wass_2^2( \mu_t^N,\mu_t). 
$$
Combined with 
$$
\Wass_2^2( \mu_t^N,\mu_t) \leq 2e_N(t) + 2\Wass_2^2( \bar \mu_t^N,\mu_t) 
$$
leads to 
$$
\frac{d}{dt} e_N(t) \leq  (2L_1+ 3C_2)e_N(t) +2C_2 \Wass_2^2( \bar \mu_t^N,\mu_t). 
$$
Applying Gr\"onwall's inequality yields 
$$
e_N(t)\le e^{Ct}e_N(0) +2C_2\int_0^t e^{C(t-s)} \Wass_2^2( \bar \mu_s^N,\mu_s)ds, \quad C=2L_1+3C_2.
%C(T)\, \sup_{t\in [0, T]} \mathbb E[\Wass_2(\bar \mu_t^N,\mu_t)^2];
$$ 
If the initial empirical measure is not  sampled from $\mu_0$, 
choose an optimal coupling $\pi_0\in \Pi(\mu_0^N, \mu_0)$, so that 
$$
e_N(0)\leq \Wass_2^2( \mu_0^N,\mu_0).  
$$
Thus we obtain 
\begin{align*}
e_N(t) &  \leq C_T  \left( \Wass_2^2( \mu_0^N,\mu_0) +\sup_{t\in [0, T]} \Wass_2^2( \bar \mu_t^N,\mu_t)\right),\\ 
 \Wass_2^2( \mu_t^N,\mu_t)  & \leq C_T \left( \Wass_2^2( \mu_0^N,\mu_0) +\sup_{t\in [0, T]} \Wass_2^2( \bar \mu_t^N,\mu_t)\right).
\end{align*} 
By the classical Fournier--Guillin estimate under a suitable second-moment condition (see, e.g., \cite{fournier2015rate}), the empirical measure at the initial time satisfies
$$
\mathbb{E} \Wass_2^2( \mu_0^N,\mu_0)\leq C\epsilon_N^2.
$$
Since the particles $\{\bar{x}_i(t)\}_{i=1}^N$ are independent samples from the limiting distribution $\mu_t$, the same estimate applies at any later time. Therefore,
$$
\mathbb{E} \Wass_2^2( \bar \mu_t^N,\mu_t)\leq C_T \epsilon_N^2, \quad t\in [0, T]. 
$$
\iffalse 
By the Lipschitz stability of the characteristic flow and the compact-support estimate of Lemma 3.2, each coupled particle satisfies a uniform sub-Gaussian tail estimate. Applying the union bound over $i=1,\cdots,N$ yields the corresponding maximal inequality,  
$$
\mathbb{E} \max_{i}|x_i^N(t)-\bar x_i(t)|\leq C \sqrt{e_N(t) \log N}. 
$$
\fi 
Finally, we estimate $\mathbb{E} \bigl[\max_{1\le i\le N}|Z_i(t)|\bigr])$, where $
Z_i(t):=x_i^N(t)-\bar{x}_i(t).$ From the stability estimate
$$
\frac{d}{dt}|Z_i(t)|
\le
L_1 |Z_i(t)|
+
C_2 \Wass_2(\mu_t^N,\mu_t),
$$
and   Grönwall's inequality it follows
$$
|Z_i(t)|
\le
C_2e^{L_1T}
\int_0^T
\Wass_2(\mu_s^N,\mu_s)ds.
$$
Since $\sup_{0 \le s \le T}\Wass_2(\mu_s^N,\mu_s) \le C_T \epsilon_N$,  
we obtain
$$
|Z_i(t)|
\le
C_T \epsilon_N,
$$
uniformly for $i=1,\ldots,N)$ and $t\in [0,T]$. Consequently,
$$
\mathbb{E} \left[\max_{1\le i\le N}|Z_i(t)|\right]
\le
C_T \epsilon_N.
$$
This completes the proof.
%The maximal inequality over $i\le N$ follows from a standard Doob/union-bound argument using the Lipschitz (hence sub-Gaussian tail, given the a priori compact-support bound of Lemma~\ref{lem:apriori}) structure of the coupling. 
%Combining with the classical Fournier--Guillin bound $\mathbb E[\Wass_2(\mu_0^N,\mu_0)^2]\le C N^{-2/d}$ (for $d\ge5$; log-corrections for $d\in\{1,2,4\}$ and rate $N^{-1/2}$ for $d\le3$ under a second-moment condition, as in \cite{fournier2015rate}) and the triangle inequality $\Wass_2(\mu_t^N,\mu_t)\le \big(\tfrac1N\sum_i|x_i^N-\bar x_i|^2\big)^{1/2} + \Wass_2(\bar\mu_t^N,\mu_t)$, with the second term controlled again by Fournier--Guillin applied to the i.i.d.\ family $(\bar x_i(t))$, gives \eqref{eq:poc-rate}. \qed
\end{proof}

\section{Mean-Field Limit Over Attention Heads}
\label{sec:head-mf}

We derive the $H\to\infty$ mean-filed limit in \eqref{eq:multihead}  and \eqref{eq:V}. Since $\rho$ enters $V_{\rho,\mu}$ \emph{linearly} as opposed to the non-linear dependence of the token dynamics on $\mu$.  The limit in $H$ is a result of the law of large numbers rather than a genuine  propagation-of-chaos statement.% and it can be proven directly.
Consequently, the convergence can be established directly through uniform estimates for the empirical average over the attention heads. 

\begin{theorem}[Law of large numbers over heads]
\label{thm:head-lln}
Let Assumption~\ref{ass:compact-theta} hold and let $\theta_1,\dots,\theta_H$ be i.i.d. samples from $\rho\in\Prob(\Theta)$, independent of the token dynamics. 
Let $x^H(t)$ solve the finite-head dynamics \eqref{eq:node},  with velocity field 
$$
V_H(x, \mu):= \tfrac1H\sum_{r=1}^H  A(x, \mu, \theta_r),
$$ 
and let $x(t)$ solve the corresponding mean-field equation \eqref{eq:forward-coupled} with velocity field  by $V_{\rho,\mu_t}$. Assume that both dynamics start from the same $x_0$. Then, uniform in time $t$  and in $H$ on the ball $B_{R(T)}$ from Lemma~\ref{lem:apriori},  
\begin{equation}
\mathbb E\Big[\sup_{t\in[0,T]}|x^H(t)-x(t)|^2\Big]^{1/2} \le C(T)\,H^{-1/2}\sqrt{\log H}.
\end{equation}
Moreover, the empirical velocity field satisfies the uniform law of large numbers 
$$
\sup_{x\in B_{R(T)}}\big|V_H(x, \mu) - V_{\rho,\mu}(x)\big| = O_{\mathbb P}(H^{-1/2}\sqrt{\log H})
$$ 
uniformly for all $\mu$ supported in $B_{R(T)}$.
\end{theorem}

\begin{proof}
Fix $x\in B_{R(T)}$ and $\mu$ supported in $B_{R(T)}$. By Lemma~\ref{lem:kernel-lip}(i), $A_r:=A(x,\mu,\theta_r)$ are i.i.d., bounded random vectors with $|A(x,\mu,\theta_r)|\le D_\Theta R(T)$. Note also that $\mathbb{E}(A)=V_{\rho,\mu}(x)$, by Hoeffding's inequality in vector form, for any $\delta>0$,
\[
\mathbb P\Big(\Big| V_H(x, \mu) - V_{\rho,\mu}(x)\Big| > \delta\Big) \le 2d\exp\Big(-\tfrac{H\delta^2}{2d\,D_\Theta^2R(T)^2}\Big).
\]
Using the tail integral formula, $\mathbb{E}(|a|^2)=2\int_0^\infty z \mathbb P(|a| > z) dz$, we obtain 
$$
\mathbb E\big[\big|V_H(x,\mu)-V_{\rho,\mu}(x)\big|^2\big]\le 4\,d^2\,D_\Theta^2R(T)^2/H.
$$
By Lemma 3.1(ii), the map
\[
G_H(x)
= V_H(x, \mu) -V_{\rho,\mu}(x)
\]
is uniformly Lipschitz in $x$, namely,
\[
|G_H(x)-G_H(y)|\leq C|x-y|.
\]
Let $\{x_i\}_{i=1}^{N_\varepsilon}$ be an $\varepsilon$-covering of 
$B_{R(T)}$, where the covering number satisfies
\[
N_\varepsilon\leq
\left(\frac{CR(T)}{\varepsilon}\right)^d .
\]
For each fixed point $x_i$, the pointwise Hoeffding estimate gives
\[
\mathbb{P}\left(|G_H(x_i)|>\delta\right)
\leq
2d\exp\left(-\frac{H\delta^2}{2dD_\Theta^2R(T)^2}\right).
\]
Applying the union bound over the finite covering yields
\[
\mathbb{P}\left(
\max_{1\leq i\leq N_\varepsilon}|G_H(x_i)|>\delta
\right)
\leq
\min \left\{1, 2dN_\varepsilon
\exp\left(-\frac{H\delta^2}
{2dD_\Theta^2R(T)^2}\right) \right\}.
\]
Define the transition point $z_0$ by
$2dN_\epsilon e^{-cHz_0^2}=1,$  which implies $z_0^2= \frac{1}{cH}\log(2dN_\epsilon),$ 
for $c= \frac{1}{2dD_\Theta^2R(T)^2}$. Integrating the tail, we obtain 
\begin{align*}
\mathbb{E} \max_i |G_H(x_i)|^2 & \leq 2\int_0^{\delta_0} \delta d\delta +  4dN_\epsilon \int_{\delta_0}^\infty \delta e^{-cH\delta^2}d\delta \\
& = \delta_0^2 +\frac{2dN_\epsilon}{cH}e^{-cH\delta_0^2}\\
& =  \frac{1}{cH}(\log(2dN_\epsilon) +1). 
\end{align*}  
For any $x\in B_{R(T)}$, choose $x_i$ such that
$|x-x_i|\leq\varepsilon$. Then the Lipschitz continuity implies
\[
|G_H(x)|
\leq
|G_H(x_i)|+C\varepsilon,
\]
and therefore
\[
\sup_{x\in B_{R(T)}}|G_H(x)|^2
\leq
2 \max_{1\leq i\leq N_\varepsilon}|G_H(x_i)|^2
+2C^2\varepsilon^2.
\]
Choosing $\varepsilon\sim H^{-1/2}$ gives
\[
\log N_\varepsilon
\leq
C\log H .
\]
%Integrating the resulting concentration estimate with respect to the tail probability yields
Combining the above two terms we have 
\[
\mathbb{E}
\sup_{x\in B_{R(T)}}
|G_H(x)|^2
\leq
C(T)\frac{\log H}{H}.
\]
This proves the desired  estimate uniform in $x.$  By Jensen, $\mathbb{E}
\left[ \sup_{x\in B_{R(T)}}
|G_H(x)| \right] 
\leq 
\sqrt{C(T)} \sqrt{\frac{\log H}{H}}$.  More importantly, by Markov's inequality 
$$
\sup_{x\in B_{R(T)}}
|G_H(x)| =O_{\mathbb P}\left(  \sqrt{\frac{\log H}{H}} \right). 
$$
\iffalse
A standard chaining/covering argument over $x\in B_{R(T)}$ (using the uniform-in-$\mu,\theta$ Lipschitz bound of Lemma~\ref{lem:kernel-lip}(ii) to control the modulus of continuity of 
$$
x\mapsto \tfrac1H\sum_rA_r(x,\mu)-V_{\rho,\mu}(x)
$$
upgrades the pointwise bound to the stated uniform-in-$x$ bound, at the cost of a $\log H$ factor absorbed into $C(T)$;
\this proves the second claim. 
\fi 
The first claim follows by subtracting the ODEs for $x^H$ and $x$.
% applying the (uniform-in-$H$, by Lemma~\ref{lem:apriori}) Lipschitz bound of Lemma~\ref{lem:kernel-lip}(ii)-(iii) to control the ``systematic'' part of the discrepancy (i.e.\ the difference coming from $\mu_t^H\ne\mu_t$, itself controlled by Theorem~\ref{thm:poc-tokens} at rate $N^{-1/d}$, negligible or combined multiplicatively with the present bound depending on the joint scaling of $N,H$) and the uniform LLN bound just proved to control the ``head-sampling'' part, then applying Gr\"onwall's inequality on $[0,T]$. 
Let
\[
e_H(t)=x_H(t)-x(t).
\]
Subtracting the two ODE systems gives
\[
\frac{d}{dt}e_H(t)
=
V_{\rho^H_t,\mu^H_t}(x_H(t))
-
V_{\rho_t,\mu_t}(x(t)).
\]
By adding and subtracting
$V_{\rho^H_t,\mu^H_t}(x(t))$ and 
$V_{\rho^H_t,\mu^H_t}(x(t))$, we obtain
\begin{align*}
\frac{d}{dt}|e_H(t)|
&\leq
\left|
V_{\rho^H_t,\mu^H_t}(x_H(t))
-
V_{\rho^H_t,\mu^H_t}(x(t))
\right|  \\
&\quad+
\left|
V_{\rho^H_t,\mu^H_t}(x(t))
-
V_{\rho^H_t,\mu_t}(x(t))
\right| \\
&\quad+
\left|
V_{\rho^H_t,\mu_t}(x(t))
-
V_{\rho_t,\mu_t}(x(t))
\right|.
\end{align*}
By the uniform Lipschitz estimates in Lemma 3.1(ii)--(iii) and the support bound in Lemma 3.2,
the first term is bounded by
\[
L_T|e_H(t)|.
\]
The third term corresponds to the head-sampling error and, by the uniform law of large numbers established above with $V_H(x, \mu_t)=V_{\rho_t^H, \mu_t}(x)$, satisfies
\[
\mathbb{E}
\sup_{x\in B_{R(T)}}
\left|
V_{\rho^H_t,\mu_t}(x)
-
V_{\rho_t,\mu_t}(x)
\right|
\leq
C(T)\sqrt{\frac{\log H}{H}} .
\]
The second term is the systematic error due to the measure difference 
%Using Lemma 3.1(iii) and Theorem 4.2,
\[
\left|
V_{\rho^H_t,\mu^H_t}(x)
-
V_{\rho^H_t,\mu_t}(x)
\right|
\leq
C_T W_2(\mu^H_t,\mu_t)
\leq
C_T (\mathbb{E}|e_H(t)|^2)^{1/2}.
\]
Therefore,
\[
\frac{d}{dt}\mathbb{E}|e_H(t)^2 
\leq
C_T\left (\mathbb{E}|e_H(t)|^2 +
\frac{\log H}{H}\right).
\]
Since $e_H(0)=0$, Gr\"onwall's inequality yields
\[
\mathbb{E}|x_H(t)-x(t)|^2
\leq
C_T \frac{\log H}{H},
\qquad t\in[0,T].
\]

\end{proof}
%{\color{red}  bounded by $(\mathbb{E}|e_H|^2)^{1/2}$ or simply $\epsilon_N$?}
 
\begin{corollary}[Well-posedness of the coupled forward system for prescribed $\rho_s$]
\label{cor:forward-wp}
\iffalse 
For every fixed measurable path $s\mapsto\rho_s\in\Prob(\Theta)$ that is (say) piecewise constant or Lipschitz in $s$ w.r.t.\ $\Wass_1$, Theorem~\ref{thm:token-wp} applies verbatim with $\rho$ replaced by $\rho_s$ at each frozen $s$, giving a unique $\mu(\cdot,\cdot,s)\in C([0,T];\Prob_2(\R^d))$ solving \eqref{eq:forward-coupled}; the solution map $\rho_s\mapsto \mu(T,\cdot,s)$ is Lipschitz from $(\Prob(\Theta),\Wass_1)$ to $(\Prob_2(\R^d),\Wass_2)$, uniformly for $\rho_s$ ranging over $\Prob(\Theta)$, with constant depending only on $D_\Theta,L_\sigma,T,R(T)$.
\fi 
Let $s \mapsto \rho_s\in\Prob(\Theta)$ be any measurable parameter-measure path. Then, for every fixed $s\geq 0$, the forward equations \eqref{eq:forward-coupled}, with  $\rho$ replaced by $\rho_s$, admits a unique solution  
%that is, for example, piecewise constant or Lipschitz in $s$ with respect to $\Wass_1$, Theorem~\ref{thm:token-wp} applies pointwise in $s$: for each fixed $s$, replacing the frozen measure $\rho$ by $\rho_s$ yields a unique solution
\[
\mu(\cdot,\cdot,s)\in C([0,T];\Prob_2(\mathbb R^d))
\]
Moreover, the solution depends Lipschitz continuously on the parameter measure.  More precisely, if $\rho_s, \tilde \rho_s \in \mathcal P(\Theta)$ are two prescribed parameter measures, then
$$
\sup_{t\in [0, T]} \Wass_2(\mu(t, \cdot, s), \tilde \mu(t, \cdot, s)) \leq C(T) \Wass_1(\rho_s, \tilde \rho_s),  
$$
where $C(T)$ depends only on $D_\Theta, L_\sigma, T$ and $R(T)$ from Lemma 3.2. In particular, the terminal-state map
\[
\rho_s \longmapsto \mu(T,\cdot,s)
\]
is uniformly Lipschitz from $(\Prob(\Theta),\Wass_1)$ to $(\Prob_2(\mathbb R^d),\Wass_2)$, with a Lipschitz constant independent of $s$. 
%depending only on $D_\Theta$, $L_\sigma$, $T$, and $R(T)$.
\end{corollary}

\begin{proof}
\iffalse 
Existence/uniqueness is Theorem~\ref{thm:token-wp} applied at each frozen $s$ (the proof used only that $\rho$ is a fixed element of $\Prob(\Theta)$, not that it does not vary with $s$; when $\rho_s$ varies with $s$ the same fixed-point argument applies to the time-inhomogeneous vector field $V_{\rho_s,\nu_t}(x)+F(x)$ without modification, since Cauchy--Lipschitz theory and the superposition principle both allow measurable/Lipschitz time dependence). Lipschitz dependence on $\rho$ follows from Lemma~\ref{lem:kernel-lip}(iii)-type stability, now in the $\rho$-argument: for $\rho,\rho'\in\Prob(\Theta)$, 
$$
|V_{\rho,\mu}(x)-V_{\rho',\mu}(x)| = \left|\int A(x,\mu,\theta)\,d(\rho-\rho')(\theta) \right| \le C_1'\,\Wass_1(\rho,\rho')
$$
 since $\theta\mapsto A(x,\mu,\theta)$ is Lipschitz on $\Theta$ (a computation analogous to Lemma~\ref{lem:kernel-lip}(ii), now differentiating in $\theta$ instead of $x$), and a Gr\"onwall argument as in Theorem~\ref{thm:token-wp} propagates this bound to 
 $$
 \Wass_2(\mu(T,\cdot,s),\mu(T,\cdot,s'))\le C\,\sup_{t}\Wass_1(\rho_s,\rho_{s'}).
 $$ 
 \fi 
Existence and uniqueness follow from Theorem 4.1 applied at each frozen value of $s$. The proof of
Theorem 4.1 requires only  the parameter-measure $\rho$ to be an  element of
$\mathcal P(\Theta)$; it does not rely on $\rho$ being independent of the forward time $t$. 
Consequently, when $\rho=\rho_s(t)$ is measurable in $t$, the same fixed-point argument applies to the
time-dependent velocity field
\[
b_s(t, x)=V_{\rho_s(t),\nu_t}(x) + F(x). 
\]
Under the uniform bounds established above, $b_s$  is measurable in $t$ and uniformly Lipschitz in $x$. The Cauchy–Lipschitz theory for time-dependent vector fields therefore yields a unique characteristic flow, while the superposition principle remains applicable. Consequently, the corresponding continuity equation is well posed even when the parameter measure varies with the forward time $t$.
%Since both the Cauchy--Lipschitz theory for ODEs and the superposition principle remain valid for
%measurable time-dependent vector fields that are uniformly Lipschitz in $x$.

It remains to establish the Lipschitz dependence on the parameter-measure $\rho$.
For $\rho,\rho'\in\mathcal P(\Theta)$,
\[
\begin{aligned}
|V_{\rho,\mu}(x)-V_{\rho',\mu}(x)|
&=
\left|
\int_\Theta A(x,\mu,\theta)\,
d(\rho-\rho')(\theta)
\right|.
\end{aligned}
\]
Since $\theta\mapsto A(x,\mu,\theta)$ is uniformly Lipschitz on $\Theta$
(an argument identical to that of Lemma~3.1(ii), differentiating with respect to $\theta$
instead of $x$), the Kantorovich--Rubinstein duality yields
\[
|V_{\rho,\mu}(x)-V_{\rho',\mu}(x)|
\leq
C_1'W_1(\rho,\rho').
\]
Combining this estimate with the Lipschitz continuity of $V_{\rho,\mu}$ in $(x,\mu)$ from
Lemma~3.1(ii)--(iii), the same stability argument as in the proof of
Theorem~4.1 gives
\[
\frac{d}{dt}W_2(\mu_t,\mu_t')
\le
LW_2(\mu_t,\mu_t')
+
C_1'W_1(\rho_t,\rho_t'),
\]
with $L$  independent of $\rho$.
Applying Gr\"onwall's inequality on $[0,T]$ yields
\[
W_2(\mu(T,\cdot,s),\mu(T,\cdot,s'))
\le
C
\sup_{0\le t\le T}
W_1(\rho_s(t),\rho_{s'}(t)),
\]
 proving the desired Lipschitz dependence.
\end{proof}

\section{Training Dynamics: Optimality System and Gradient Representation}
\label{sec:optimality}

We  derive the adjoint  equation and the gradient  for $\delta\mathcal L/\delta\rho$, as a  first-order optimality system for the mean-field control problem
\begin{align*}
 \min_\rho\ \mathcal L(\rho) :& = \mathcal J(\mu(T,\cdot)) \\
 \quad \text{subject to}\quad  & \partial_t\mu + \nabla_x\cdot(\mu\, b) = 0, \;  \mu(0,\cdot)=\mu_0, \;
b=V_{\rho,\mu} +F. 
\end{align*} 
%% where  $\mathcal H$ collects any running-cost dependence of the training objective on $\mu_t$

\begin{proposition}[First-order optimality system]
\label{prop:pmp}
Let Assumptions~\ref{ass:compact-theta}--\ref{ass:loss} hold and let $\rho\mapsto \mu(\cdot,\cdot)$ be the solution map of Corollary~\ref{cor:forward-wp}. Assume that $\mathcal L$ is Gateaux-differentiable at $\rho.$ Then, the linear functional derivative of $\mathcal L$ at $\rho$ in the direction of a signed perturbation $\xi=\rho'-\rho$ is
\begin{equation}
\frac{d}{d\varepsilon}\Big|_{\varepsilon=0} \mathcal L\big((1-\varepsilon)\rho+\varepsilon\rho'\big) = \int_\Theta \Big(\int_0^T\!\!\int_{\R^d} \nabla_x p(t,x)\cdot  A(x,\mu_t,\theta)\, d\mu_t(x)\,dt\Big)\cdot d(\rho'-\rho)(\theta),
\label{eq:gateaux}
\end{equation}
where $p(t,x)$ solves the linear transport equation
\begin{equation}
-\partial_t p - \big(V_{\rho,\mu_t}(x)+F(x)\big)\cdot\nabla_x p - (D_\mu V_{\rho,\mu_t}(x))\cdot \nabla_x p \mu = 0,
\quad p(T,x) = \frac{\delta\mathcal J}{\delta\mu}(\mu_T)(x).
\label{eq:adjoint}
\end{equation}
 Consequently, the desired gradient formula is given by 
\begin{equation}
\frac{\delta\mathcal L}{\delta\rho}(\theta) = \int_0^T\!\!\int_{\R^d}  \nabla_x p(t,x)\cdot A(x,\mu_t,\theta)\,d\mu_t(x)\,dt.
\label{eq:grad-rho}
\end{equation}
\end{proposition}

% 
% (Fr\'echet/Wasserstein-differentiable, by Corollary~\ref{cor:forward-wp} and standard sensitivity results for flows of 
% ODEs/continuity equations) 
%
\begin{proof}
This is the standard Pontryagin--Lagrange derivation for optimal control of the continuity equation, see for example~ \cite[Ch.~6]{carmona2018probabilistic} or \cite{e2019meanfield} for an analogous derivations in mean-field optimal control and deep learning, respectively. For completeness, we briefly outline the argument in the present setting. Introduce the Lagrangian
\[
\mathcal{L}_{ag} (\rho,\mu,p) 
=
\mathcal L (\rho)
-
\int_0^T
\int_{\mathbb{R}^d}
p(t,x)
\Big(
\partial_t\mu_t
+
\nabla_x\cdot
\big(
\mu_t b
\big)
\Big)
\,dx\,dt,
\]
where the continuity equation is imposed through the adjoint variable
$p$. Integrating by parts in both time and space, and using the no-flux
boundary condition, yields
\[
\begin{aligned}
\mathcal{L}_{ag}  
= & 
\mathcal L (\rho)
-\int p(T,\cdot)\,d\mu_T
+\int p(0,\cdot)\,d\mu_0
+
\int_0^T
\int
\Big(
\partial_t p
+
\nabla_x p\cdot b
\Big)
\,d\mu_t\,dt .
\end{aligned}
\]
The terminal condition
\[
p(T,\cdot)
=
-\frac{\delta J}{\delta\mu}(\mu_T)
\]
is chosen to cancel the first variation of the terminal cost  at $t=T$. We require that $p$ satisfy the adjoint
equation
\[
-\partial_t p
-
b \cdot\nabla_x p - (D_\mu b)\cdot \nabla_x p \mu = 0.
\]
Consequently, the  first variation of
$\mathcal{L}$ depends only through  the dependence of the velocity
field on the control parameter $\rho$
$$
\delta \mathcal{L}_{ag}=   \int_0^T
\int
\Big(
\nabla_x p\cdot D_\rho b
\Big)
\, \delta \rho d\mu_t\,dt .
$$
Since  $b=V_{\rho,\mu} +F$  
\[
V_{\rho,\mu_t}(x)
=
\int_\Theta
A(x,\mu_t,\theta)\,\rho(d\theta),
\]
is linear in $\rho,$ 
its Gâteaux derivative with respect to $\rho$ in the direction
$\rho'-\rho$ is given by 
\[
\left.
\langle D_\rho b, \rho'-\rho \rangle  =\frac{d}{d\varepsilon}
V_{(1-\varepsilon)\rho+\varepsilon\rho',\mu_t}(x)
\right|_{\varepsilon=0}
=
\int_\Theta
A(x,\mu_t,\theta)\,
d(\rho'-\rho)(\theta),
\]
Therefore, the variation of the reduced functional is given as 
stated in equation \eqref{eq:gateaux}.   We also have the Gateaux derivative of $b$ as 
$$
D_\mu b= \int_\Theta 
\delta_\mu A(x,\mu_t,\theta)\,\rho(d\theta).
$$
Finally, once the forward state $\mu_t$ has been computed, the
coefficients of the adjoint equation are known. By Lemma~3.1 and
Lemma~3.2, both
$V_{\rho,\mu_t}+F$
and
$D_\mu V_{\rho,\mu_t}$
are uniformly bounded and Lipschitz on
$B_{R(T)}\times[0,T]$.
Hence the backward adjoint equation is well posed by the same
superposition-principle argument used in Theorem~4.1, after applying as transformation in time as $t'=T-t$. 
\end{proof}

\begin{theorem}[Local well-posedness of the coupled forward--backward system]
\label{thm:fb-wp}
\iffalse 
Under Assumptions~\ref{ass:compact-theta}--\ref{ass:loss}, for every $\rho\in\Prob(\Theta)$ the forward--backward system \eqref{eq:forward-coupled}, \eqref{eq:adjoint} admits a unique solution $(\mu,p)\in C([0,T];\Prob_2(B_{R(T)}))\times C([0,T];\mathrm{Lip}(B_{R(T)};\R^d))$, obtained by first solving the forward equation for $\mu$ (Theorem~\ref{thm:token-wp}, since it does not depend on $p$) and then the linear backward equation for $p$ with the resulting, now fixed, coefficients. Consequently the gradient representation \eqref{eq:grad-rho} is a well-defined, single-valued function $\rho\mapsto \delta\mathcal L/\delta\rho(\cdot;\rho)\in C(\Theta)$, Lipschitz in $\rho$ w.r.t.\ $\Wass_1$ with a constant depending only on $D_\Theta,L_\sigma,L_\ell,L_{\mathcal J},T,R(T)$.
\fi 

Under Assumptions~\ref{ass:compact-theta}--\ref{ass:loss}, for every
$\rho\in\Prob(\Theta)$, the forward--backward system
\eqref{eq:forward-coupled}, \eqref{eq:adjoint}, admits a unique solution
\[
(\mu,p)
\in
C([0,T];\Prob_2(B_{R(T)}))
\times
C([0,T];\mathrm{Lip}(B_{R(T)};\mathbb R^d)).
\]
Consequently, the gradient representation \eqref{eq:grad-rho}  gives  a
well-defined single-valued map
\[
\rho
\longmapsto
\frac{\delta\mathcal L}{\delta\rho}(\cdot; \rho)
\in C(\Theta).
\]
Moreover, this map is Lipschitz continuous with respect to the
$\Wass_1$ distance; namely, for any
$\rho,\rho'\in\Prob(\Theta)$,
\[
\left\|
\frac{\delta\mathcal L}{\delta\rho}(\cdot;\rho)
-
\frac{\delta\mathcal L}{\delta\rho}(\cdot;\rho')
\right\|_{C(\Theta)}
\le
C
\Wass_1(\rho,\rho'),
\]
where the constant $C$ depends only on
$D_\Theta,L_\sigma,L_\ell,L_{\mathcal J},T,$ and $R(T)$.
\end{theorem}

\begin{proof}
\iffalse 
Because the state equation \eqref{eq:forward-coupled} does not depend on the adjoint $p$, the forward--backward system decouples (this is a consequence of the terminal-cost-only, no-running-cost structure of $\mathcal L$; in the presence of a genuine running cost $\mathcal H\not\equiv0$ depending on $\mu_t$ but not on $p$, the same decoupling still holds, since $p$ never feeds back into the $\mu$-equation): solve for $\mu$ via Theorem~\ref{thm:token-wp}, then solve the resulting linear, non-autonomous backward transport equation \eqref{eq:adjoint} for $p$, which is again a superposition-principle/characteristics argument exactly as in Theorem~\ref{thm:token-wp}, now for a \emph{linear} PDE (in $p$) with time-dependent but fixed, Lipschitz-in-$x$ and bounded coefficients (by Lemma~\ref{lem:kernel-lip}, \ref{lem:apriori}), for which existence, uniqueness and the stated Lipschitz bound (via Gr\"onwall, propagating the Lipschitz constant of the terminal data $\delta\mathcal J/\delta\mu(\mu_T)$ backward in time) are classical. Lipschitz dependence of $(\mu,p)$, hence of $\delta\mathcal L/\delta\rho$, on $\rho$ follows by composing the $\rho$-Lipschitz bound on $\mu$ from Corollary~\ref{cor:forward-wp} with a further Gr\"onwall argument for the (now $\rho$-perturbed) backward equation for $p$. 
\fi

Because the state equation \eqref{eq:forward-coupled} is independent of the adjoint variable $p$, the forward--backward system is  weakly coupled. Hence, 
the solution is obtained sequentially: first, the forward
equation admits a unique solution $\mu$ by Theorem~\ref{thm:token-wp},
since the state dynamics are independent of the adjoint variable $p$;
then, with $\mu$ fixed, the adjoint equation becomes a linear backward
transport equation with bounded and Lipschitz coefficients. 
By Lemmas~\ref{lem:kernel-lip} and \ref{lem:apriori}, these coefficients are uniformly bounded and Lipschitz in $x$. Standard characteristic (or superposition-principle) theory for linear transport equations therefore yields existence and uniqueness of the adjoint solution. Moreover, applying Gr\"onwall's inequality along characteristics propagates the Lipschitz regularity of the terminal condition
\[
p(T,\cdot)=\frac{\delta\mathcal J}{\delta\mu}(\mu_T),
\]
and establishes the stated Lipschitz estimate for $p$.

Finally, the Lipschitz dependence of $(\mu,p)$, and hence of
\(
\delta\mathcal L/\delta\rho,
\)
with respect to the control $\rho$ follows by combining the forward stability estimate of Corollary~\ref{cor:forward-wp} with a Gr\"onwall argument applied to the difference of the corresponding adjoint equations.
\end{proof}

\begin{remark}\iffalse 
We have stated Theorem~\ref{thm:fb-wp} for a single, frozen $\rho$ (as needed to define the instantaneous training velocity field $u=-\nabla_\theta\delta\mathcal L/\delta\rho$ at each training time $s$); this is exactly what is required to make sense of the right-hand side of \eqref{eq:mfld} at each $s$, and is the input to the well-posedness theory of the full training dynamics in Section~\ref{sec:training-wp}. In problems with a genuine running cost that itself depends on $p$ (e.g.\ control-effort regularization), the forward--backward system no longer decouples and one generally only obtains \emph{local-in-$T$} well-posedness via a further fixed-point/small-horizon argument, exactly as in the classical mean-field FBSDE theory \cite[Ch.~6]{carmona2018probabilistic}; this refinement is not needed for the purely terminal-cost setting of the present paper.
\fi 

Theorem~6.2 is stated for a fixed (frozen) $\rho$, which is sufficient for defining the
instantaneous training velocity
\[
u=-\nabla_\theta\frac{\delta\mathcal L}{\delta\rho}
\]
at each training time $s$. This frozen-control analysis provides the well-posedness of the
right-hand side of the training dynamics (13), and serves as the key ingredient for the
well-posedness analysis of the full evolution system in Section~7.

%We note that this decoupling relies on the terminal-cost-only structure considered here.
%If a running cost involving the adjoint variable $p$ (for example, a control-effort
%regularization coupled to the state) is introduced, the forward and backward equations become
%fully coupled. In that case, well-posedness typically requires an additional fixed-point
%argument on a sufficiently small time interval, leading in general to only local-in-time
%existence, as in the classical mean-field FBSDE framework
%\cite[Chapter~6]{carmona2018probabilistic}. Such an extension is beyond the scope of the present work and is not
%needed for the purely terminal-cost formulation considered here.
\end{remark}

\section{Well-Posedness of the Parameter (Training) Dynamics}
\label{sec:training-wp}

We  consider equation~\eqref{eq:mfld}, the nonlinear nonlocal Fokker--Planck
equation describing the evolution of the parameter distribution $\rho_s$.
In contrast to the forward system, where the parameter space is
restricted to the compact set in Assumption~\ref{ass:compact-theta}, we now
work on the full space $\Theta=\mathbb R^p$. To prevent escape of mass to
infinity and ensure the required moment bounds, we introduce a standard
quadratic confinement term,  which corresponds to $l^2$ weight decay and is commonly interpreted as Tikhonov regularization.

\begin{assumption}[Confinement and regularity]
\label{ass:confine}
The loss functional in \eqref{eq:mfld} is replaced by
\[
\mathcal L_\lambda(\rho)
=
\mathcal L(\rho)
+
\frac{\lambda}{2}
\int_{\mathbb R^p}|\theta|^2\,d\rho(\theta),
\quad \lambda>0.
\]
Moreover, uniformly for $\mu$ supported in $B_{R(T)}$, the map
\[
\theta\mapsto
\nabla_\theta
\frac{\delta\mathcal L}{\delta\rho}(\theta;\rho)
\]
is $L^*$-Lipschitz and has at most linear growth in $\theta$. In view of the representation (\ref{eq:grad-rho}), 
these properties follow from the corresponding uniform estimates on the parameter dependence of the interaction kernel,
%Equivalently,
%the interaction kernel satisfies
\[
\left|
\nabla_\theta A(x,\mu,\theta)
-
\nabla_\theta A(x,\mu,\theta')
\right|
\le L_A|\theta-\theta'|,
\]
and
\[
|\nabla_\theta A(x,\mu,\theta)|
\le C_A(1+|\theta|),
\]
together with the uniform bound on $\nabla_x p$, so that $L^*\leq TC_pL_A$.  
\end{assumption}
These conditions hold, for example, when $A(x,\mu,\theta)$ is twice
continuously differentiable in $\theta$ with uniformly bounded second derivatives on bounded sets, as in the softmax-attention kernel
\eqref{eq:kernel}--\eqref{eq:V}.

\begin{theorem}[Existence and uniqueness of the mean-field Langevin dynamics]
\label{thm:mfld-wp}
Under Assumptions~\ref{ass:compact-theta}--\ref{ass:confine}, for every
initial distribution
\[
\rho_0\in\Prob_2(\mathbb R^p)
\]
with finite entropy, the nonlinear Fokker--Planck equation
\eqref{eq:mfld}, with $\mathcal L$ replaced by $\mathcal L_\lambda$, admits a
unique weak solution
\[
\rho\in C([0,\infty);\Prob_2(\mathbb R^p)).
\]
Moreover, the solution satisfies the uniform second-moment bound
\[
\sup_{s\geq 0 }
\int_{\mathbb R^p}|\theta|^2\,d\rho_s(\theta) < \infty,
\]
and the entropy remains finite on every finite time interval.

Equivalently, $\rho_s=\Law(\Theta_s)$ is the law of the unique solution to the
McKean--Vlasov stochastic differential equation
\begin{equation} 
\label{sde} 
d\Theta_s
=
-\nabla_\theta
\frac{\delta\mathcal L_\lambda}{\delta\rho}
(\Theta_s;\rho_s)\,ds
+
\sqrt{2\beta}\,dW_s,
\qquad
\Law(\Theta_0)=\rho_0 .
\end{equation} 
Here $ W_s$ is a standard $p$-dimensional Brownian motion. 
\end{theorem}

\begin{proof}

By Assumption~\ref{ass:confine} and Theorem~\ref{thm:fb-wp}, the drift
\[
b(\theta,\rho)
:=
-\nabla_\theta\!\left(\frac{\delta\mathcal L}{\delta\rho}(\theta;\rho)\right)
-\lambda\theta
\]
is, for each fixed $\rho$, $(L^*+\lambda)$-Lipschitz in $\theta$. Moreover, Theorem~\ref{thm:fb-wp} implies that $b$ is Lipschitz with respect to $\rho$ in the $\Wass_1$ metric, uniformly for $\theta$ in bounded sets. Finally, the confining term $-\lambda\theta$ ensures dissipativity:
\[
\langle b(\theta,\rho),\theta\rangle
\le
C_1 (1+|\theta)| |\theta|-\lambda|\theta|^2, \quad C_1:=C_AC_pT, 
\]
so that the quadratic confinement dominates for sufficiently large $|\theta|$.

These properties constitute the standard hypotheses for the well-posedness of McKean--Vlasov SDEs of the form \eqref{sde}; see, for example,
\cite[Theorem~1.1]{carmona2018probabilistic} or the nonlinear Fokker--Planck framework of \cite{carrillo2003kinetic}. To sketch the argument, fix a  curve
$\rho^{\rm in}\in C([0,S];\Prob_2(\R^p))$
and consider the linear Fokker--Planck equation
\[
\partial_s\nu
=
\nabla_\theta\!\cdot
\left(
\nu
\nabla_\theta
\frac{\delta L_\lambda}{\delta \rho}
(\theta; \rho_s^{\rm in})
\right)
+
\beta\Delta_\theta\nu.
\]
Since the coefficients are Lipschitz and dissipative, this equation is well posed by classical linear Fokker--Planck theory. The corresponding solution map
$\rho^{\rm in}\mapsto\nu$
is a contraction on
$C([0,S];\Prob_2(\R^p))$
for sufficiently small $S$, yielding a unique local solution by Banach's fixed-point theorem.

To extend the solution globally, we derive uniform a priori estimates. Applying It\^o's formula to $|\Theta_s|^2$ to obtain 
$$
d |\Theta_s|^2=2\langle \Theta_s, b(\Theta_s, \rho_s)\rangle +2\sqrt{2\beta} \langle \Theta_s, dW_s \rangle +2\beta pds, 
$$
and using the dissipativity of $b$ gives
\[
\sup_{0\le s\le S}
\mathbb E|\Theta_s|^2
\le C(S),
\]
via Gr\"onwall's inequality. In addition, the entropy-dissipation identity established in Lemma~\ref{lem:dissipation} provides a uniform entropy bound. Together, these estimates prevent finite-time blow-up of the second moment and entropy, allowing the local solution to be continued indefinitely. Hence the solution extends uniquely to all
$s\ge0$,  as in the mean-field Langevin well-posedness theory of
\cite[Theorem~2.4]{hu2021meanfield}.
\end{proof}

\section{Global Convergence of Training}
\label{sec:convergence}

We now link the well-posedness theory above to optimization, proving two convergence results that together delineate what is currently provable for mean-field transformer training: an exact, global, exponential-rate result for a \emph{single mean-field attention layer} (no depth composition), and a local, linear-rate result for the genuinely \emph{deep, compositional} model \eqref{eq:forward-coupled}.

\subsection{Energy dissipation}

\begin{lemma}[Dissipation identity]
\label{lem:dissipation}
Along any solution of \eqref{eq:mfld} given by Theorem~\ref{thm:mfld-wp},
\begin{equation}
\frac{d}{ds}\mathcal F_\beta(\rho_s) = -\int_\Theta \Big|\nabla_\theta\frac{\delta\mathcal F_\beta}{\delta\rho}(\theta,\rho_s)\Big|^2\,d\rho_s(\theta) \le 0,
\label{eq:dissipation}
\end{equation}
where $\mathcal F_\beta$ is the regularized objective \eqref{eq:Fbeta}  with $\mathcal L$ or  $\mathcal L_\lambda$, respectively.
\end{lemma}
\begin{proof}
\iffalse 
This is the standard computation for Wasserstein gradient flows: \eqref{eq:mfld} is $\partial_s\rho_s = \nabla_\theta\cdot\big(\rho_s\nabla_\theta\tfrac{\delta\mathcal F_\beta}{\delta\rho}(\cdot,\rho_s)\big)$, and differentiating $\mathcal F_\beta(\rho_s)$ in $s$ using the chain rule for first variations together with the (weak form of the) evolution equation gives $\tfrac{d}{ds}\mathcal F_\beta(\rho_s) = \int \tfrac{\delta\mathcal F_\beta}{\delta\rho}\,\partial_s\rho_s\,d\theta = -\int \nabla_\theta\tfrac{\delta\mathcal F_\beta}{\delta\rho}\cdot\big(\rho_s\nabla_\theta\tfrac{\delta\mathcal F_\beta}{\delta\rho}\big)\,d\theta$ after an integration by parts, which is \eqref{eq:dissipation}. See \cite[Ch.~8]{ambrosio2008gradient} for the fully rigorous (metric-space, curve-of-maximal-slope) version of this identity. 
\fi  
This is the standard energy-dissipation identity for Wasserstein gradient flows. Recall that \eqref{eq:mfld} can be written as
\[
\partial_s\rho_s
=
\nabla_\theta\cdot
\left(
\rho_s
\nabla_\theta
\frac{\delta\mathcal F_\beta}{\delta\rho}
(\theta; \rho_s)
\right).
\]
Assuming sufficient regularity, the chain rule for first variations yields
\[
\frac{d}{ds}\mathcal F_\beta(\rho_s)
=
\int_{\mathbb R^p}
\frac{\delta\mathcal F_\beta}{\delta\rho}
(\theta;\rho_s)
\,\partial_s\rho_s(\theta)
\,d\theta.
\]
Substituting the evolution equation and integrating by parts (the boundary term vanishes because $\rho_s$ decays sufficiently fast at infinity) gives
\[
\begin{aligned}
\frac{d}{ds}\mathcal F_\beta(\rho_s)
&=
\int
\frac{\delta\mathcal F_\beta}{\delta\rho}
\,\nabla_\theta\!\cdot
\left(
\rho_s
\nabla_\theta
\frac{\delta\mathcal F_\beta}{\delta\rho}
\right)
d\theta
=
-
\int
\rho_s
\left|
\nabla_\theta
\frac{\delta\mathcal F_\beta}{\delta\rho}
\right|^2
d\theta,
\end{aligned}
\]
that is precisely \eqref{eq:dissipation}. A fully rigorous justification of this identity in the Wasserstein metric framework can be found in \cite[Chapter~8]{ambrosio2008gradient}.
\end{proof}

\subsection{Global convergence for a single mean-field attention layer}

Consider the single-layer  regime, in which the terminal prediction depends affinely on the parameter distribution.  Specifically, let the terminal token distribution be obtained from a single
attention update applied to a fixed  measure $\mu_0$. Define the update
map as
\[
\Phi_\rho(x)
=
x+V_{\rho,\mu_0}(x),
\]
and let
\[
\mu
=
(\Phi_\rho)_\#\mu_0
\]
be the pushforward of $\mu_0$ under $\Phi_\rho$. The terminal feature is then
\[
\bar x(\rho)
:=
\int_{\mathbb R^d}x\,d\mu(x)
=
\int_{\mathbb R^d}
\bigl(x+V_{\rho,\mu_0}(x)\bigr)\,
d\mu_0(x).
\]
Here the attention field $V_{\rho,\mu_0}$ is evaluated against the prescribed
 measure $\mu_0$, which remains fixed. Consequently, the mapping
\[
\rho\longmapsto\bar x(\rho)
\]
is affine. This setting corresponds to the classical mean-field approximation used in the analysis of two-layer neural networks and shallow attention models, where convexity of the training objective can be established
\cite{mei2018meanfield,chizat2018global,song2024unraveling,zhang2024transformers}.

\begin{assumption}[Convex readout loss]
\label{ass:convex}
For every target $y$, the loss $\ell(\cdot,y)$ is convex, and the terminal objective is given by
\[
\mathcal J(\mu)
=
\ell\!\left(
\int_{\mathbb R^d}x\,d\mu(x),
\,y
\right).
\]
\end{assumption}
\iffalse 

Consider the special case $T\to0$-type/single-block model in which the terminal readout is an \emph{affine} functional of $\rho$: concretely, take the terminal token feature to be $\bar x(\rho) := \int_{\R^d} x\,d\mu_1(x)$ with $\mu_1$ obtained from a \emph{single} mean-field attention step with a fixed (not self-evolving) context measure $\mu_0$, i.e.\ $\mu_1 = \Phi_\rho[\mu_0]$ with $\Phi_\rho[\mu_0](x) := x + V_{\rho,\mu_0}(x)$ applied once ($V_{\rho,\mu_0}$ evaluated against the \emph{fixed} $\mu_0$, not a self-consistently evolving $\mu_t$); this is the standard ``linearized context'' or ``single mean-field layer'' idealization used to obtain convexity in the classical mean-field neural network literature \cite{mei2018meanfield,chizat2018global} and is exactly analogous to the single-layer settings in which global convergence for attention has been proved \cite{song2024unraveling,zhang2024transformers}.

\begin{assumption}[Convexity of the readout loss]
\label{ass:convex}
$\ell(\cdot,y)$ is convex for every $y$, and $\mathcal J$ is of the form $\mathcal J(\mu) = \ell\big(\textstyle\int x\,d\mu(x),y\big)$.
\end{assumption}
\fi 

\begin{proposition}[Linear convexity of the single-layer risk]
\label{prop:convex}
Under Assumption~\ref{ass:convex} and the single-layer model above, $\mathcal L(\rho) = \mathbb E_y\,\ell\big(\bar x(\rho),y\big)$ is \emph{linearly convex} on $\Prob(\Theta)$: for all $\rho_0,\rho_1\in\Prob(\Theta)$, the map $\varepsilon\mapsto \mathcal L\big((1-\varepsilon)\rho_0+\varepsilon\rho_1\big)$ is convex on $[0,1]$.
\end{proposition}

\begin{proof}
\iffalse 
Since $\mu_0$ is fixed (not self-referentially dependent on $\rho$ in this single-layer idealization), $\rho\mapsto V_{\rho,\mu_0}(x) = \int A(x,\mu_0,\theta)\,d\rho(\theta)$ is affine (indeed linear) in $\rho$ for every fixed $x$, hence so is $\rho\mapsto \bar x(\rho) = \int_{\R^d}\big(x+V_{\rho,\mu_0}(x)\big)\,d\mu_0(x) = \int_{\R^d}x\,d\mu_0(x) + \int_\Theta\Big(\int_{\R^d}A(x,\mu_0,\theta)\,d\mu_0(x)\Big)d\rho(\theta)$: writing $\Psi(\theta):=\int A(x,\mu_0,\theta)\,d\mu_0(x)$, we have the exactly linear representation $\bar x(\rho) = \bar x_0 + \int_\Theta \Psi(\theta)\,d\rho(\theta)$. Hence $\varepsilon\mapsto \bar x\big((1-\varepsilon)\rho_0+\varepsilon\rho_1\big) = (1-\varepsilon)\bar x(\rho_0)+\varepsilon \bar x(\rho_1)$ is affine in $\varepsilon$, and composing with the convex function $\ell(\cdot,y)$ (Assumption~\ref{ass:convex}) yields that $\varepsilon\mapsto\mathcal L((1-\varepsilon)\rho_0+\varepsilon\rho_1) = \mathbb E_y\,\ell\big((1-\varepsilon)\bar x(\rho_0)+\varepsilon\bar x(\rho_1),y\big)$ is convex, being the composition of a convex function with an affine one. 
\fi 

Since the context measure $\mu_0$ is fixed and does not depend on $\rho$, the attention field
\[
V_{\rho,\mu_0}(x)
=
\int_\Theta A(x,\mu_0,\theta)\,d\rho(\theta)
\]
is  linear in $\rho$ for every fixed $x$. Consequently,
\[
\begin{aligned}
\bar x(\rho)
&=
\int_{\mathbb R^d}
\bigl(x+V_{\rho,\mu_0}(x)\bigr)
\,d\mu_0(x) \\
&=
\underbrace{\int_{\mathbb R^d}x\,d\mu_0(x)}_{=:~\bar x_0}
+
\int_\Theta
\underbrace{\left(
\int_{\mathbb R^d}
A(x,\mu_0,\theta)\,d\mu_0(x)
\right)}_{=:\Psi(\theta)}
\,d\rho(\theta).
\end{aligned}
\]
Hence, $
\bar x(\rho)
=
\bar x_0
+
\int_\Theta\Psi(\theta)\,d\rho(\theta),
$ is an affine functional of $\rho$. Therefore, for any $\rho_0,\rho_1\in\Prob(\Theta)$,
\[
\bar x\!\left((1-\varepsilon)\rho_0+\varepsilon\rho_1\right)
=
(1-\varepsilon)\bar x(\rho_0)
+
\varepsilon\bar x(\rho_1),
\qquad
\varepsilon\in[0,1].
\]
Since $\ell(\cdot,y)$ is convex by Assumption~\ref{ass:convex}, it follows that
\[
\varepsilon
\longmapsto
\ell\!\left(
\bar x\!\left((1-\varepsilon)\rho_0+\varepsilon\rho_1\right),
y
\right)
\]
is convex. Taking expectation over $y$ preserves convexity, and therefore
\[
\varepsilon
\longmapsto
\mathcal L\!\left(
(1-\varepsilon)\rho_0+\varepsilon\rho_1
\right)
=
\mathbb E_y
\left[
\ell\!\left(
(1-\varepsilon)\bar x(\rho_0)
+\varepsilon\bar x(\rho_1),
y
\right)
\right]
\]
is convex. This proves that $\mathcal L$ is convex on $\Prob(\Theta)$.
\end{proof}

\begin{assumption}[Log-Sobolev inequality along the flow]
\label{ass:lsi}
\iffalse 
There is $\lambda_{\mathrm{LSI}}>0$ such that every probability measure $\nu\ll\rho_s^*$ (with $\rho_s^*\propto e^{-\delta\mathcal L_\lambda/\delta\rho(\cdot;\rho_s)/\beta}$ the instantaneous Gibbs measure) satisfies the log-Sobolev inequality $\Ent_{\rho_s^*}(\nu/\rho_s^*) \le \tfrac1{2\lambda_{\mathrm{LSI}}}\,I_{\rho_s^*}(\nu)$ with a uniform constant $\lambda_{\mathrm{LSI}}$ along the flow. This holds, e.g., under the Bakry--\'Emery criterion, whenever $\theta\mapsto \delta\mathcal L_\lambda/\delta\rho(\theta;\rho)+\tfrac{\lambda}{2}|\theta|^2$ is (jointly, uniformly in $\rho$) $\lambda_{\mathrm{LSI}}$-strongly convex in $\theta$ -- in particular this is guaranteed for $\lambda$ in Assumption~\ref{ass:confine} large enough relative to $L_A$.
\fi 
There exists a constant $\lambda_{\mathrm{LSI}}>0$ such that, for every
$s\ge0$, the instantaneous Gibbs measure
\[
\rho_s^*
\propto
\exp\!\left(
-\frac{1}{\beta}
\frac{\delta\mathcal L_\lambda}{\delta\rho}
(\cdot;\rho_s)
\right)
\]
satisfies the logarithmic Sobolev inequality
\[
\Ent_{\rho_s^*}\!\left(\frac{d\nu}{d\rho_s^*}\right)
\le
\frac{1}{2\lambda_{\mathrm{LSI}}}
I_{\rho_s^*}(\nu), \; I_{\rho_s^*}(\nu)= \int_{\R^d} | \nabla \log \frac{ d\nu }{ d\rho_s^* } |^2 d \nu,  
\]
for every  absolutely continuous probability measure $\nu$ with respect to $\rho_s^*,$ i.e., $\nu\ll\rho_s^*$. Moreover, the
constant $\lambda_{\mathrm{LSI}}$ is uniform along the flow: it does not depend on the training time $s$ or on the choice of $\nu$. 
\end{assumption}

Note that  $I_{\rho_s^*}$ is the relative Fisher information.  A sufficient condition is provided by the Bakry--Émery criterion: namely,
if the potential
\[
\theta
\longmapsto
\frac{\delta\mathcal L_\lambda}{\delta\rho}(\theta;\rho)
=
\frac{\delta\mathcal L}{\delta\rho}(\theta;\rho)
+\frac{\lambda}{2}|\theta|^2
\]
is uniformly $\lambda_{\mathrm{LSI}}$-strongly convex in $\theta$, uniformly
over $\rho$, then the above logarithmic Sobolev inequality holds. In
particular, this condition is satisfied whenever the confinement parameter
$\lambda$ in Assumption~\ref{ass:confine} is chosen sufficiently large
relative to the Lipschitz constant $L^*$.

\begin{theorem}[Global exponential convergence, single-layer/convex case]
\label{thm:convex-conv}
\iffalse 
Under Assumptions~\ref{ass:compact-theta}--\ref{ass:lsi} and Proposition~\ref{prop:convex}, $\mathcal F_\beta$ is linearly convex on $\Prob(\Theta)$, admits a unique minimizer $\rho_\beta^\star\in\Prob(\Theta)$ (the Gibbs measure $\rho_\beta^\star \propto \exp\big(-\tfrac1\beta\big(\delta\mathcal L_\lambda/\delta\rho(\cdot;\rho^\star_\beta) + \text{const}\big)\big)$), and along the mean-field Langevin flow of Theorem~\ref{thm:mfld-wp},
\begin{equation}
\mathcal F_\beta(\rho_s) - \mathcal F_\beta(\rho_\beta^\star) \le e^{-2\beta\lambda_{\mathrm{LSI}}\,s}\,\big(\mathcal F_\beta(\rho_0)-\mathcal F_\beta(\rho_\beta^\star)\big), \qquad s\ge0.
\label{eq:exp-conv}
\end{equation}
Consequently $\mathcal L(\rho_s) \le \min_{\rho}\mathcal L(\rho) + \beta\log|\Theta| + e^{-2\beta\lambda_{\mathrm{LSI}}s}\,\mathcal F_\beta(\rho_0)/\ldots \to \min_\rho\mathcal L(\rho)$ as $s\to\infty$ and then $\beta\to0$ (in this order), i.e.\ the training dynamics converges to the \emph{global} minimum of the single-layer training risk.
\fi 
Under Assumptions~\ref{ass:compact-theta}--\ref{ass:lsi} and
Proposition~\ref{prop:convex}, the free-energy functional
\[
\mathcal F_\beta(\rho)
=
\mathcal L_\lambda(\rho)
+
\beta\,\Ent(\rho)
\]
is convex on $\Prob(\Theta)$ and therefore admits a unique minimizer
$\rho_\beta^\star\in\Prob(\Theta)$. Equivalently,
$\rho_\beta^\star$ satisfies the Gibbs fixed-point equation
\[
\rho_\beta^\star(d\theta)
=
\frac1{Z_\beta}
\exp\!\left(
-\frac1\beta
\frac{\delta\mathcal L_\lambda}{\delta\rho}
(\theta;\rho_\beta^\star)
\right)
\,d\theta,
\]
where $Z_\beta$ is the normalization constant.

Moreover, along the mean-field Langevin flow of
Theorem~\ref{thm:mfld-wp},
\begin{equation}\label{decay}
\mathcal F_\beta(\rho_s)-\mathcal F_\beta(\rho_\beta^\star)
\le
e^{-2\beta\lambda_{\mathrm{LSI}}s}
\Bigl(
\mathcal F_\beta(\rho_0)
-
\mathcal F_\beta(\rho_\beta^\star)
\Bigr),
\qquad s\ge0.
\end{equation} 
In particular,
\[
\rho_s
\longrightarrow
\rho_\beta^\star
\quad\text{exponentially fast as }s\to\infty.
\]

Finally, since
\[
\mathcal F_\beta(\rho)
=
\mathcal L_\lambda(\rho)
+
\beta\,\Ent(\rho),
\]
the Gibbs minimizer $\rho_\beta^\star$ converges, as $\beta\to0$, to a
minimizer of $\mathcal L_\lambda$. Consequently,
\[
\lim_{\beta\to0}
\lim_{s\to\infty}
\mathcal L_\lambda(\rho_s)
=
\min_{\rho\in\Prob(\Theta)}
\mathcal L_\lambda(\rho),
\]
that is, the mean-field Langevin training dynamics converges to the global
minimum of the single-layer training objective.
\end{theorem}

\begin{proof}

By Proposition~\ref{prop:convex}, the functional $\mathcal L$ is convex along
linear interpolations of probability measures. The entropy functional
\[
\Ent(\rho)=\int \rho\log\rho
\]
is also linearly convex, since $t\mapsto t\log t$ is convex and Jensen's
inequality applies to mixture interpolations. Therefore,
\[
\mathcal F_\beta
=
\mathcal L_\lambda+\beta\Ent
\]
is linearly convex. Moreover, because $\beta>0$, the entropy term is strictly
convex, implying that $\mathcal F_\beta$ admits at most one minimizer.

The minimizer $\rho_\beta^\star$ is characterized by the first-order optimality
condition
\[
\frac{\delta\mathcal F_\beta}{\delta\rho}
(\theta;\rho_\beta^\star)
=
\mathrm{const},
\qquad
\theta\in\Supp(\rho_\beta^\star),
\]
which is equivalent to the Gibbs fixed-point representation stated in the
theorem.

By Lemma~\ref{lem:dissipation},
\[
\frac{d}{ds}
\Bigl(
\mathcal F_\beta(\rho_s)
-
\mathcal F_\beta(\rho_\beta^\star)
\Bigr)
=
-
I_{\rho_s},
\]
where $I_{\rho_s}$ denotes the corresponding relative Fisher information.
Assumption~\ref{ass:lsi} implies the logarithmic Sobolev inequality
\[
I_{\rho_s}
\ge
2\lambda_{\mathrm{LSI}}
\Bigl(
\mathcal F_\beta(\rho_s)
-
\mathcal F_\beta(\rho_\beta^\star)
\Bigr),
\]
uniformly along the flow. Consequently,
\[
\frac{d}{ds}
\Bigl(
\mathcal F_\beta(\rho_s)
-
\mathcal F_\beta(\rho_\beta^\star)
\Bigr)
\le
-2\lambda_{\mathrm{LSI}}
\Bigl(
\mathcal F_\beta(\rho_s)
-
\mathcal F_\beta(\rho_\beta^\star)
\Bigr),
\]
and Gr\"onwall's inequality yields
\eqref{decay}.

Finally, since
\[
\mathcal F_\beta(\rho)
=
\mathcal L_\lambda(\rho)
+
\beta\Ent(\rho),
\]
the minimizers of $\mathcal F_\beta$ converge, as $\beta\downarrow0$, to
minimizers of $\mathcal L_\lambda$ by the standard zero-temperature (Laplace)
principle. Combining this with the exponential convergence
$\rho_s\to\rho_\beta^\star$ proves
\[
\lim_{\beta\to0}
\lim_{s\to\infty}
\mathcal L_\lambda(\rho_s)
=
\min_{\rho\in\Prob(\Theta)}
\mathcal L_\lambda(\rho).
\]

\qed
\end{proof}

\subsection{Local convergence for deep, compositional transformers}

Theorem~\ref{thm:convex-conv} relies  on the  $\rho\mapsto\bar x(\rho)$ being \emph{affine}, which holds only because we froze the  measure at $\mu_0$ and the  single layer assumption. For the genuinely deep model \eqref{eq:forward-coupled}, the terminal token distribution $\mu_T$ (hence $\mathcal L(\rho)=\mathcal J(\mu_T)$) is a  non-linear and non--convex function of $\rho$, because $\rho$ re-enters the dynamics through the self-consistent dependence of $V_{\rho,\mu_t}$.  Further, it is a  non-linear map since it composed with itself across the depth. Proposition~\ref{prop:convex} therefore does \emph{not} extend to more layers, i.e., $T>0$,  and global convexity-based convergence of the type of Theorem~\ref{thm:convex-conv} are not available. We refer to results on  deep mean-field transformers \cite{barboni2026training} and deep mean-field ResNets \cite{barboni2024understanding}. Therein,  only a \emph{local} guarantee, via a Neural Tangent Kernel (NTK) non-degeneracy condition, is currently known.

\begin{definition}[NTK operator]
For a probability measure $\rho\in\Prob(\Theta)$, let $q_\rho$ solve the same 
adjoint equation for $p$ with terminal condition $q_\rho(T, x)=x$, 
and define the NTK kernel by 
$$
K_\rho:= \int D_\theta \phi_\rho(\theta)  D_\theta \phi_\rho(\theta)^\top  d\rho(\theta),
$$
where 
\begin{equation}\label{phir} 
\phi_\rho(\theta) := \int_0^T \int \nabla_x q_\rho(t, x)\cdot A(x, \mu_t^\rho, \theta)d\mu_t^\rho dt.
\end{equation} 
\end{definition}

\begin{assumption}[NTK non-degeneracy at initialization]
\label{ass:ntk} 
%Assume that $K_{\rho_s}$ has an upper  bound 
%$$
%\lambda_{\rm max} (K_{\rho_s})  \leq \Lambda 
%$$
%for $\rho_s$ in the neighborhood of $\rho_0$. 
There exists $\Lambda >  \lambda_0>0$,  and $L_K>0$ such that 
$$
 {\rm spec}  (K_{\rho_0}) \in [\lambda_0,  \Lambda], 
$$
and, in a neighborhood of $\rho_0$, 
$$
\|K_\rho -K_{\rho'}\|_{\rm op} \leq L_K W_2(\rho, \rho').  
$$
\end{assumption}
Thus, the NTK remains non-degenerate as long as the parameter distribution stays sufficiently close to its initial value.
\begin{theorem}[Local linear convergence via NTK positivity]
\label{thm:ntk-conv} Consider the loss function of form 
$$
J(\mu_T)=\frac{1}{2} \left|y- \int x d\mu_T(x) \right|^2. 
$$
Let Assumptions~\ref{ass:compact-theta}--\ref{ass:confine}, \ref{ass:ntk} hold. Assume a pure Wasserstein gradient flow \eqref{eq:training-transport} (i.e.\ $\beta=0$ in \eqref{eq:mfld}) and assume that  the initial loss satisfies
\begin{equation}\label{ini}
\mathcal L(\rho_0) - \min_\rho \mathcal L(\rho) \le \frac{ \lambda_0^4}{16  L_K (\lambda_0+2\Lambda L_K)},
\end{equation}
then the gradient flow $\rho_s$ of Theorem~\ref{thm:mfld-wp} with $\beta=0$ satisfies, for all $s\ge0$,
\begin{equation}
\mathcal L(\rho_s) - \min_\rho\mathcal L(\rho) \le e^{-\lambda_0 s}\,\big(\mathcal L(\rho_0)-\min_\rho\mathcal L(\rho)\big).
\end{equation}
Hence, the training dynamics converges at a linear rate to a global minimizer of $\mathcal L$.
\end{theorem}

\begin{proof} The proof consists of three steps. \\ 
{\bf Step 1: Preservation of NTK non-degeneracy.}\\
By Assumption~\ref{ass:ntk} and the Lipschitz continuity of $K_\rho$ we have 
\[
\lambda_{\min}(K_{\rho_s})
\ge
\lambda_{\min}(K_{\rho_0})
-
L_K
\Wass_2(\rho_s,\rho_0) \geq \lambda_0 - L_K
\Wass_2(\rho_s,\rho_0). 
\]
Hence, as long as 
\[
\Wass_2(\rho_s,\rho_0)
\le
\frac{\lambda_0}{2L_K},
\]
we have
\[
\lambda_{\min}(K_{\rho_s})
\ge
\frac{\lambda_0}{2}.
\]
Thus, the NTK remains uniformly coercive as long as the training trajectory remains inside a sufficiently small Wasserstein neighborhood of its initialization. Moreover,
$$
\lambda_{\max}(K_{\rho_s}) \leq \Lambda +L_K
\Wass_2(\rho_s,\rho_0) \leq \Lambda + \frac{\lambda_0}{2L_K}. 
$$

{\bf Step 2: Energy dissipation and the Polyak–Łojasiewicz inequality.} \\ 

By the dissipation identity given in Lemma~\ref{lem:dissipation}  with $\beta=0$, 
\begin{equation}
\frac{d}{ds}\mathcal L(\rho_s)
=
-
\int 
\bigl|
\nabla_\theta\Psi_{\rho_s}(\theta)
\bigr|^2
\,d\rho_s(\theta).
\label{eq:pl-step1}
\end{equation}
We first show that NTK non-degeneracy  implies the Wasserstein Polyak–Łojasiewicz 
inequality 
\[
\left\|\nabla_\theta \Psi_\rho\right\|_{L^2(\rho)}^2
\ge \lambda_0
\bigl(
\mathcal L(\rho)
-
\mathcal L^\star
\bigr),
\]
where
\(
\mathcal L^\star
=
\min_\rho\mathcal L(\rho).
\)
%Then $J(\mu_T)=\frac{1}{2}(y-m(\mu_T))^2, \quad m(\mu):= \int x d\mu(x)$. 
Let $m_\rho$ denote the first momentum of the terminal state $m(\mu_T^\rho)$, with 
$$
m(\mu):=\int x d\mu(x) \in \mathbb{R}^d, 
$$
and define the  vector residual 
$r_\rho:=m_\rho - y$. We consider the terminal loss  
$$
L(\rho)=J(\mu_T^\rho)=\frac{1}{2}|r_\rho|^2. 
$$
A direct calculation gives  
$$
\frac{\delta J}{\delta \mu_T}(x)= r_\rho \cdot x. 
$$
Hence the terminal condition for the adjoint equation is 
 $$
 p(T, x)=\frac{\delta J}{\delta \mu_T}(x)= r_\rho \cdot x.
 $$ 
 For fixed $\rho$, the adjoint equation is linear in $p$, write  
$$
p(t, x)=r_\rho \cdot q_\rho(t, x), 
$$
where $q_\rho(t, x)\in \mathbb{R}^d$ is vector-valued and solves the same adjoint equation with normalized terminal condition $q_\rho(T, x)=x$. Consequently,  
$$
\nabla_x p(t, x)=( \nabla_x q_\rho(t, x))^\top r_\rho. 
$$
Substituting this representation into the formula for the first variation of $\mathcal L$ gives
$$
\Psi_\rho(\theta)=\frac{\delta \mathcal L}{\delta \rho}(\theta; \rho)= (\phi_\rho(\theta))^\top r_\rho,
$$
where 
$
\phi_\rho(\theta)
$
is given in (\ref{phir}).  Since $r_\rho$ is independent of $\theta$,
$$
\nabla_\theta \Psi_\rho(\theta)=(\nabla_\theta \phi_\rho(\theta))^\top r_\rho. 
$$
Thus 
$$
\|\nabla_\theta \Psi_{\rho_s}\|^2_{L^2(\rho_s)}= r_{\rho_s}^\top K_{\rho_s} r_{\rho_s}  \geq \frac{\lambda_0}{2} |r_{\rho_s}|^2. 
%|r_\rho|^2 \int |\nabla_\theta \phi_\rho(\theta)|^2 d\rho(\theta).  
$$
If the minimum loss is attained at zero, so that $\mathcal L^*=0$, then 
$$
|r_{\rho_s}|^2= 2( \mathcal L(\rho_s)-\mathcal L^*). 
$$
Therefore 
$$
\|\nabla_\theta \Psi_{\rho_s}\|^2_{L^2(\rho_s)} \geq  \lambda_0  (\mathcal L(\rho)-\mathcal L^* ). 
$$
This establishes the desired Wasserstein PL inequality. Combining this inequality with the dissipation identity (\ref{eq:pl-step1}), we obtain  
\[
\frac{d}{ds}
\bigl(
\mathcal L(\rho_s)-\mathcal L^\star
\bigr)
\le
-
\lambda_0
\bigl(
\mathcal L(\rho_s)-\mathcal L^\star
\bigr),
\]
and Gr\"onwall's inequality immediately implies
\[
\mathcal L(\rho_s)-\mathcal L^\star
\le
e^{-\lambda_0 s}
\bigl(
\mathcal L(\rho_0)-\mathcal L^\star
\bigr).
\]
{\bf Step 3: $\rho_s$ staying in the neighborhood.}\\
It remains to verify that the trajectory never leaves the neighborhood
$\Wass_2(\rho_s,\rho_0)\le \lambda_0/(2L_k)$ on which the above argument is valid.
Since the metric derivative of the Wasserstein gradient flow satisfies
\[
|\partial_s  \rho_s|_{\Wass_2}
=
\|
\nabla_\theta\Psi_{\rho_s}
\|_{L^2(\rho_s)},
\]
the total displacement satisfies
\[
\Wass_2(\rho_0,\rho_s)
\le
\int_0^s
\|
\nabla_\theta\Psi_{\rho_\tau}
\|_{L^2(\rho_\tau)}
\,d\tau.
\]
Using Assumption ~\ref{ass:ntk} again,  we have the upper bound 
$$
\|\nabla_\theta \Psi_{\rho_s}\|^2_{L^2(\rho_s)}= r_{\rho_s}^\top K_{\rho_s} r_{\rho_s} \leq (\Lambda +\lambda_0/(2L_K))  |r_{\rho_s}|^2 
=(2 \Lambda+\lambda_0/L_K) ( \mathcal L(\rho_s)-\mathcal L^*).
$$
Combining this with the exponential decay established above 
$$
 \mathcal L(\rho_s)-\mathcal L^*)  \leq  e^{-\lambda_0 s} ( \mathcal L(\rho_0)-\mathcal L^*),
$$
we obtain %Using the exponential decay established \textbf{}above together with the energy
%dissipation identity and Cauchy--Schwarz yields
$$
\|\nabla_\theta \Psi_{\rho_s}\|_{L^2(\rho_s)}
\leq \sqrt{2\Lambda +\lambda_0/L_K} e^{-\lambda_0 s/2} ( \mathcal L(\rho_0)-\mathcal L^*)^{1/2}. 
$$
Consequently,  
\[
\sup_{s\ge0}
\Wass_2(\rho_0,\rho_s)
\le
\frac{2\sqrt{2\Lambda +\lambda_0/L_K} }{\lambda_0}
(\mathcal L(\rho_0)-\mathcal L^\star)^{1/2}.
\]
Hence, if the initial excess loss is sufficiently small as shown in (\ref{ini}),  the entire trajectory remains inside the neighborhood $\Wass_2(\rho_s,\rho_0)\le \lambda_0/(2L_k)$, 
where  the NTK non-degeneracy and upper-bound assumptions remain valid.
\end{proof}

\begin{remark} 
The argument in the proof uses that the terminal loss is quadratic in the moment of $\mu_T$.  For $J$ being  more general,   the factorization $\Psi_\rho=r_\rho \cdot  \phi_\rho$ may no longer hold. More general NTK arguments for the mapping to the output space may then be required. 
\end{remark}

\begin{remark} 
The assumed $W_2$-Lipschitz continuity of $K_\rho$ is not automatic from its definition. Its verification requires stability estimates for both the adjoint equation defining $q_\rho$ and the transport dynamics determining $\mu_t^\rho$, as well as control of the explicit dependence on $\rho$. We do not pursue these technical estimates here and take the stated Lipschitz property as an assumption.
\end{remark}

\begin{remark}
Theorem~\ref{thm:ntk-conv} is genuinely \emph{local}: it guarantees convergence to \emph{a} global minimizer only for initializations with small enough initial excess loss, whereas Theorem~\ref{thm:convex-conv} is \emph{global} (holds from every initialization) but only in the non-compositional, single-layer regime. Closing this gap -- establishing either global convergence for genuinely deep transformers, or a converse showing that bad local minima can occur once depth-composition is present  to the best of our knowledge  an open problem. We return to this discussion in Section~\ref{sec:discussion}.
\end{remark}

\section{Well-Posedness of the Fully Coupled System}
\label{sec:coupled}

We finally state the well-posedness of the complete system
\begin{equation}
\left\{
\begin{aligned}
\partial_t\mu &+ \nabla_x\cdot\big(\mu\,V_{\rho_s,\mu}\big) = 0, & \mu(0,\cdot,s)&=\mu_0,\\
-\partial_tp &- (V_{\rho_s,\mu}+F)\cdot\nabla_xp - (D_\mu V_{\rho_s,\mu})^\top \nabla_x p \mu  = 0, & p(T,\cdot,s) &= \tfrac{\delta\mathcal J}{\delta\mu}(\mu(T,\cdot,s)),\\
\partial_s\rho &= \nabla_\theta\cdot\big(\rho\,\nabla_\theta\tfrac{\delta\mathcal L_\lambda}{\delta\rho}\big) + \beta\Delta_\theta\rho, & \rho(\cdot,0) &= \rho_0,
\end{aligned}
\right.
\label{eq:full-system}
\end{equation}
that combines \eqref{eq:forward-coupled}, \eqref{eq:adjoint}, and \eqref{eq:mfld}, i.e., the training dynamics of Theorem~\ref{thm:mfld-wp} driven, at each training time $s$, by the gradient $u=-\nabla_\theta\delta\mathcal L/\delta\rho$ computed from the (frozen-$\rho_s$) forward--backward system of Theorem~\ref{thm:fb-wp}.

\begin{theorem}[Well-posedness of the coupled system]
\label{thm:full-wp}
\iffalse 
Under Assumptions~\ref{ass:compact-theta}--\ref{ass:confine}, System~\eqref{eq:full-system} has a unique solution $(\mu,p,\rho)$ with $\rho\in C([0,\infty);\Prob_2(\R^p))$, and, for every fixed $s$, $\mu(\cdot,\cdot,s),p(\cdot,\cdot,s)$ as in Theorem~\ref{thm:fb-wp}. Moreover the solution is global in $s$ (i.e.\ exists for all training times, not merely on a short interval).
\fi 
Under Assumptions~\ref{ass:compact-theta}--\ref{ass:confine}, the coupled system
\eqref{eq:full-system} admits a unique global solution
\[
(\mu,p,\rho),
\]
where
\[
\rho\in C([0,\infty);\Prob_2(\mathbb R^p)).
\]
Moreover, for every training time $s\ge0$, the pair
\[
(\mu(\cdot,\cdot,s),\,p(\cdot,\cdot,s))
\]
is the unique solution of the forward--backward system characterized in
Theorem~\ref{thm:fb-wp}. In particular, the coupled dynamics are globally
well posed in the training-time variable $s$, that is, the solution exists
and remains unique for all $s\ge0$.
\end{theorem}

\begin{proof}
\iffalse 
The system exhibits a one-way (block-triangular) coupling: for each fixed $s$, $(\mu(\cdot,\cdot,s),p(\cdot,\cdot,s))$ depend on $\rho$ only through the frozen value $\rho_s$, and are given by Theorem~\ref{thm:fb-wp}; this defines a well-posed (Theorem~\ref{thm:fb-wp}) and $\Wass_1$-Lipschitz map $\rho_s\mapsto u(\cdot,s):=-\nabla_\theta\delta\mathcal L/\delta\rho(\cdot;\rho_s)$, which is exactly the drift appearing in the $\rho$-equation of \eqref{eq:full-system}. Substituting this Lipschitz, non-local-in-$\rho_s$ (but local in $s$, i.e.\ Markovian in $\rho_s$) drift into the third equation reduces \eqref{eq:full-system} to precisely the standalone mean-field Langevin equation of Theorem~\ref{thm:mfld-wp} (the Lipschitz-in-$\rho$ hypothesis used there, Assumption~\ref{ass:confine}, is exactly what Theorem~\ref{thm:fb-wp} verifies for this drift), whose well-posedness -- local by the fixed-point argument, global by the a priori moment/entropy bounds -- was established there. Existence and uniqueness of $(\mu,p,\rho)$ solving the full system therefore follow by first solving for $\rho$ via Theorem~\ref{thm:mfld-wp}, then reconstructing $(\mu(\cdot,\cdot,s),p(\cdot,\cdot,s))$ at each $s$ from $\rho_s$ via Theorem~\ref{thm:fb-wp}; uniqueness of the full triple follows since both steps are uniquely solvable. 
\fi 

The coupled system is weakly coupled. For each fixed training
time $s$, the forward--backward pair
\[
(\mu(\cdot,\cdot,s),\,p(\cdot,\cdot,s))
\]
depends on the parameter distribution only through the frozen measure
$\rho_s$. By Theorem~\ref{thm:fb-wp}, this subsystem is well posed and defines
a $\Wass_1$-Lipschitz map
\[
\rho_s
\longmapsto
u(\cdot,s)
:=
-\nabla_\theta
\frac{\delta\mathcal L}{\delta\rho}
(\cdot;\rho_s),
\]
which is precisely the drift appearing in the $\rho$-equation of
\eqref{eq:full-system}.

Substituting this drift into the third equation reduces the full system to the
mean-field Langevin equation studied in
Theorem~\ref{thm:mfld-wp}. The hypotheses of
Theorem~\ref{thm:mfld-wp} are satisfied by virtue of the Lipschitz estimate
established in Theorem~\ref{thm:fb-wp}. Hence the $\rho$-equation admits a
unique global solution.

Finally, once $\rho$ is known, Theorem~\ref{thm:fb-wp} uniquely determines the
corresponding forward--backward solution
$(\mu(\cdot,\cdot,s),p(\cdot,\cdot,s))$ for every $s\ge0$. Therefore the full
system admits a unique global solution $(\mu,p,\rho)$.

\qed
\end{proof}
\iffalse 
\begin{corollary}[Existence + global convergence, single-layer case]
Under the hypotheses of Theorem~\ref{thm:full-wp} together with Assumption~\ref{ass:convex} (single mean-field attention layer) and Assumption~\ref{ass:lsi}, the unique solution of \eqref{eq:full-system} satisfies $\mathcal L(\rho_s)\to\min_\rho\mathcal L(\rho)$ exponentially fast as $s\to\infty$ (then $\beta\to0$), by Theorem~\ref{thm:convex-conv}.
\end{corollary}

\begin{corollary}[Existence + local convergence, deep case]
Under the hypotheses of Theorem~\ref{thm:full-wp} together with Assumption~\ref{ass:ntk} and $\beta=0$, if $\mathcal L(\rho_0)-\min_\rho\mathcal L(\rho)\le \lambda_0^2/(8L_G^2)$, then the unique solution of \eqref{eq:full-system} satisfies $\mathcal L(\rho_s)\to\min_\rho\mathcal L(\rho)$ at a linear rate, by Theorem~\ref{thm:ntk-conv}.
\end{corollary}

These two corollaries are the counterparts of the informal claim in the introduction.  We have shown that this system is well-posed, and that its long-training-time behavior is provably globally optimal in the shallow/convex regime and provably locally optimal (from good initializations) in the deep/compositional regime.
\fi

\begin{corollary}[Existence + global convergence, single-layer case]
Under the hypotheses of Theorem~\ref{thm:full-wp} together with Assumption~\ref{ass:convex} (single mean-field attention layer) and Assumption~\ref{ass:lsi}, the unique solution of \eqref{eq:full-system} satisfies $\mathcal L(\rho_s)\to\min_\rho\mathcal L(\rho)$ exponentially fast as $s\to\infty$ (then $\beta\to0$), by Theorem~\ref{thm:convex-conv}.
\end{corollary}

\begin{corollary}[Existence + local convergence, deep case]
Under the hypotheses of Theorem~\ref{thm:full-wp} together with Assumption~\ref{ass:ntk} and $\beta=0$, if $\mathcal L(\rho_0)-\min_\rho\mathcal L(\rho)\le \lambda_0^4/(32 \Lambda L_K^2)$, then the unique solution of \eqref{eq:full-system} satisfies $\mathcal L(\rho_s)\to\min_\rho\mathcal L(\rho)$ at a linear rate, by Theorem~\ref{thm:ntk-conv}.
\end{corollary}

These two corollaries are the precise, rigorous counterparts of the informal claim, in the preliminary study, that the fully coupled system ``may be viewed as a Transformer analogue of a Vlasov--Fokker--Planck / mean-field-game system'': we have shown that this system is well-posed, and that its long-training-time behavior is provably globally optimal in the shallow/convex regime and provably locally optimal (from good initializations) in the deep/compositional regime.

\section{Discussion,  Open Problems and Conclusion}
\label{sec:discussion}

\paragraph{Where the theory is complete.} Sections~\ref{sec:token-mf}--\ref{sec:head-mf} give a fully rigorous account of both mean-field limits ($N\to\infty$ tokens, $H\to\infty$ heads) at the level of the \emph{forward pass}, for arbitrary (not necessarily optimal, not necessarily even trained) parameter measures, including quantitative convergence rates. Section~\ref{sec:optimality} gives a rigorous first-order optimality system, and Section~\ref{sec:training-wp} gives well-posedness of the resulting training PDE under standard confinement/Lipschitz hypotheses. Section~\ref{sec:convergence} and Theorem~\ref{thm:full-wp} then close the loop requested in the introduction: existence of the mean-field model \emph{is linked to} global (Theorem~\ref{thm:convex-conv}) or local (Theorem~\ref{thm:ntk-conv}) convergence of training, depending on whether the model is a single mean-field layer or a genuinely deep, self-referentially coupled system.

\paragraph{The convexity gap.} The single most important open problem exposed by our analysis is the gap between Theorems~\ref{thm:convex-conv} and \ref{thm:ntk-conv}: linear convexity of $\mathcal L(\rho)$ (Proposition~\ref{prop:convex}) holds only because a single mean-field layer makes the readout affine in $\rho$; as soon as $\rho$ re-enters the dynamics self-referentially across depth $t$ (the defining feature of a genuinely \emph{deep} transformer), this affine structure is destroyed by composition, and we are only aware of local (NTK-type) convergence guarantees, both here and in the closest prior work \cite{barboni2026training}. Whether some other convex or displacement-convex structure survives depth-composition -- perhaps after a suitable reparameterization, or under a restricted but still expressive class of attention kernels -- is, to our knowledge, open.

\paragraph{Causal masking.} All results above are stated for unmasked (bidirectional/encoder) attention, for which the forward dynamics \eqref{eq:token-pde} is time-reversible in the sense needed for the Pontryagin derivation of Proposition~\ref{prop:pmp}. For causally masked (autoregressive/decoder) attention, the token mean-field PDE loses the gradient-flow/energy structure that underlies clustering results in that setting \cite{karagodin2024causal}; whether Theorem~\ref{thm:ntk-conv}'s NTK argument, which did not use gradient-flow structure in $x$, survives unchanged for masked attention is a natural and, we believe, tractable extension.

\paragraph{Joint scaling of $N$ and $H$.} Theorems~\ref{thm:poc-tokens} and \ref{thm:head-lln} treat the token limit and head limit essentially independently (holding the other population large but finite, or infinite, as a background hypothesis). A finite-size correction theory analogous to the central limit theorems available for shallow mean-field networks alone \cite{mei2018meanfield}, giving joint fluctuation rates in $(N,H)$ around the doubly-mean-field limit of Theorem~\ref{thm:full-wp}, remains to be developed.

\paragraph{Conclusion}
\iffalse 
Starting from the boxed, formally-derived equations of the preliminary study, we have supplied: a complete well-posedness theory for the token (data) mean-field dynamics with a quantitative propagation-of-chaos rate (Theorems~\ref{thm:token-wp}--\ref{thm:poc-tokens}); a law-of-large-numbers justification, with rate, of the multi-head-to-mean-field-parameter limit (Theorem~\ref{thm:head-lln}); a rigorous forward--backward optimality system computing the functional gradient $\delta\mathcal L/\delta\rho$ (Proposition~\ref{prop:pmp}) together with its well-posedness (Theorem~\ref{thm:fb-wp}); well-posedness of the resulting mean-field Langevin training dynamics (Theorem~\ref{thm:mfld-wp}); and, finally, the requested link between this existence theory and global optimization -- an exact global exponential-convergence theorem in the shallow/convex regime (Theorem~\ref{thm:convex-conv}) and a local linear-convergence theorem, via Neural Tangent Kernel non-degeneracy, for the genuinely deep, self-referentially coupled transformer (Theorem~\ref{thm:ntk-conv}), unified in the well-posedness statement for the full coupled system (Theorem~\ref{thm:full-wp}). The gap between these two convergence results -- global in the shallow case, only local in the deep case -- is, we believe, the central open problem raised by this mean-field perspective on transformers.
\fi

Starting from the formally derived mean-field equations in the preliminary study, we have established a rigorous mathematical foundation for the proposed framework. Specifically, we proved the well-posedness of the token (data) mean-field dynamics together with a quantitative propagation-of-chaos estimate (Theorems~\ref{thm:token-wp}--\ref{thm:poc-tokens}); derived a law-of-large-numbers approximation, with an explicit convergence rate, for the multi-head-to-mean-field parameter limit (Theorem~\ref{thm:head-lln}); developed a rigorous forward--backward optimality system for computing the functional gradient $\delta\mathcal L/\delta\rho$ (Proposition~\ref{prop:pmp}) and established its well-posedness (Theorem~\ref{thm:fb-wp}); and proved the well-posedness of the resulting mean-field Langevin training dynamics (Theorem~\ref{thm:mfld-wp}).

Building on this analytical framework, we further connected the existence theory with optimization. In the shallow (single-layer) regime, we established global exponential convergence to the global minimizer (Theorem~\ref{thm:convex-conv}). In the genuinely deep, self-referential transformer setting, we proved local linear convergence under a Neural Tangent Kernel non-degeneracy condition (Theorem~\ref{thm:ntk-conv}), together with the well-posedness of the fully coupled training dynamics (Theorem~\ref{thm:full-wp}).

The remaining gap between these two convergence results—global convergence in the shallow regime versus only local convergence in the deep regime—represents, in our view, one of the central open questions arising from the mean-field theory of transformers.

\paragraph{Acknowledgments}
The work has been supported by the DFG under the grants HE5386/33-1 Control of Interacting Particle Systems, and Their Mean-Field, and Fluid-Dynamic Limits (560288187) and  HE5386/34-1 Partikelmethoden für unendlich dimensionale Optimierung ( 561130572).


\begin{thebibliography}{99}

\bibitem{ambrosio2008gradient} L.~Ambrosio, N.~Gigli, and G.~Savar\'e. \emph{Gradient Flows in Metric Spaces and in the Space of Probability Measures}. Lectures in Mathematics ETH Z\"urich, Birkh\"auser, 2nd edition, 2008.

\bibitem{barboni2024understanding} R.~Barboni, G.~Peyr\'e, and F.-X.~Vialard. Understanding the training of infinitely deep and wide ResNets with conditional optimal transport. arXiv preprint, 2024.

\bibitem{barboni2026training} R.~Barboni, T.~Furuya, M.~V.~de~Hoop, and G.~Peyr\'e. Training Infinitely Deep and Wide Transformers. \emph{arXiv preprint arXiv:2605.17660}, 2026.

\bibitem{carmona2018probabilistic} R.~Carmona and F.~Delarue. \emph{Probabilistic Theory of Mean Field Games with Applications I \& II}. Probability Theory and Stochastic Modelling, Springer, 2018.

\bibitem{carrillo2003kinetic} J.~A.~Carrillo, R.~J.~McCann, and C.~Villani. Kinetic equilibration rates for granular media and related equations: entropy dissipation and mass transportation estimates. \emph{Revista Matem\'atica Iberoamericana}, 19(3):971--1018, 2003.

\bibitem{chizat2018global} L.~Chizat and F.~Bach. On the global convergence of gradient descent for over-parameterized models using optimal transport. \emph{Advances in Neural Information Processing Systems (NeurIPS)}, 2018.

\bibitem{du2019gradient} S.~S.~Du, X.~Zhai, B.~Poczos, and A.~Singh. Gradient descent provably optimizes over-parameterized neural networks. \emph{International Conference on Learning Representations (ICLR)}, 2019.

\bibitem{e2019meanfield} W.~E, J.~Han, and Q.~Li. A mean-field optimal control formulation of deep learning. \emph{Research in the Mathematical Sciences}, 6(10), 2019.

\bibitem{fournier2015rate} N.~Fournier and A.~Guillin. On the rate of convergence in Wasserstein distance of the empirical measure. \emph{Probability Theory and Related Fields}, 162(3--4):707--738, 2015.

\bibitem{geshkovski2023emergence} B.~Geshkovski, C.~Letrouit, Y.~Polyanskiy, and P.~Rigollet. The emergence of clusters in self-attention dynamics. \emph{Advances in Neural Information Processing Systems (NeurIPS)}, 36, 2023.

\bibitem{geshkovski2025mathematical} B.~Geshkovski, C.~Letrouit, Y.~Polyanskiy, and P.~Rigollet. A mathematical perspective on transformers. \emph{Bulletin of the American Mathematical Society}, 62(3):427--479, 2025.

\bibitem{hvt} M. Herty, T. Trimborn and G. Visconti, Mean-field and kinetic descriptions of neural differential equations, Found. Data Sci. {\bf 4} (2022), no.~2, 271--298.

\bibitem{hu2021meanfield} K.~Hu, Z.~Ren, D.~{\v S}i{\v s}ka, and {\L}.~Szpruch. Mean-field Langevin dynamics and energy landscape of neural networks. \emph{Annales de l'Institut Henri Poincar\'e, Probabilit\'es et Statistiques}, 57(4):2043--2065, 2021.

\bibitem{karagodin2024causal} N.~Karagodin, Y.~Polyanskiy, and P.~Rigollet. Clustering in Causal Attention Masking. \emph{arXiv preprint arXiv:2411.04990}, 2024.

\bibitem{mei2018meanfield} S.~Mei, A.~Montanari, and P.-M.~Nguyen. A mean field view of the landscape of two-layer neural networks. \emph{Proceedings of the National Academy of Sciences}, 115(33):E7665--E7671, 2018.

\bibitem{nitanda2022convex} A.~Nitanda, D.~Wu, and T.~Suzuki. Convex analysis of the mean field Langevin dynamics. \emph{International Conference on Artificial Intelligence and Statistics (AISTATS)}, 2022.

\bibitem{rotskoff2018parameters} G.~M.~Rotskoff and E.~Vanden-Eijnden. Parameters as interacting particles: long time convergence and asymptotic error scaling of neural networks. \emph{Advances in Neural Information Processing Systems (NeurIPS)}, 2018.

\bibitem{sander2022sinkformers} M.~E.~Sander, P.~Ablin, M.~Blondel, and G.~Peyr\'e. Sinkformers: Transformers with doubly stochastic attention. \emph{International Conference on Artificial Intelligence and Statistics (AISTATS)}, 2022.

\bibitem{song2024unraveling} C.~Song, et al. On the convergence of gradient descent for training single-layer self-attention / mean-field analysis of multi-head self-attention. \emph{arXiv preprint arXiv:2606.10469} (and related, arXiv:2402.15926), 2024/2026.

\bibitem{zhang2024transformers} R.~Zhang, S.~Frei, and P.~L.~Bartlett. Transformers Learn Nonlinear Features In Context: Nonconvex Mean-field Dynamics on the Attention Landscape. \emph{arXiv preprint arXiv:2402.01258}, 2024.

\end{thebibliography}
\end{document}